\RequirePackage{snapshot}
\documentclass[final]{siamltex}
\usepackage{microtype}
\usepackage[text={5.125in,8.25in},centering]{geometry}
\usepackage{amsmath}
\usepackage{amssymb}
\usepackage{graphicx}
\usepackage{url}
\usepackage{booktabs}
\usepackage{nicefrac}
\usepackage[colorlinks,citecolor={blue},urlcolor={blue},linkcolor={blue}]{hyperref}
\usepackage{dsfont}
\usepackage{enumitem}
\usepackage{appendix}
\usepackage{mdframed}
\usepackage{braket}
\usepackage{xspace}

\input{./macros}

\mdfdefinestyle{mystyle}{%
    outerlinewidth  = 2             ,%
    leftmargin      = 40            ,%
    rightmargin     = 40            ,%
    backgroundcolor = softblue      ,%
    outerlinecolor  = blue          ,%
    innertopmargin  = .5\topskip      ,%
    splittopskip    = 1.1\topskip   ,%
    splitbottomskip = 0.5\topskip,
    ntheorem        = false          ,%
    skipabove       = \baselineskip ,%
    skipbelow       = \baselineskip]
}

\newmdtheoremenv[style=mystyle]{btheorem}[theorem]{Theorem}
\newmdtheoremenv[style=mystyle]{blemma}[theorem]{Lemma}
\newmdtheoremenv[style=mystyle]{bcorollary}[theorem]{Corollary}
\newmdtheoremenv[style=mystyle,backgroundcolor=white]{bassumption}[theorem]{Assumption}
\newmdtheoremenv[style=mystyle,backgroundcolor=white]{bhypothesis}[theorem]{Hypothesis}

\usepackage{boxedminipage}
\usepackage{color}
\definecolor{softblue}{rgb}{0.90,0.92,1.00}

\newcommand{\dom}[1]{\mathop{\mathrm{dom}} #1}

\newtheorem{example}[theorem]{Example}
\newcommand{\mgf}{\gamma}

\newcommand{\Ex}{\mathop{\bf E\/}}
\newcommand{\Soln}{\ensuremath{\Sscr}\xspace}
\newcommand{\error}[1]{\smash{\norm{e_{#1}}^2}}
\newcommand{\bmgf}{\bar\gamma} \newcommand{\calG}{\mathcal{G}}
\newcommand{\ind}[1]{\mathds{1}_{#1}} \renewcommand{\mid}{\,|\,}
\DeclareMathOperator{\prox}{\rm prox}
\title{Tail bounds for stochastic approximation} \author{%
  Michael P.~Friedlander%
  \thanks{Department of Computer Science, University of British
    Columbia, Vancouver, BC V6T 1Z4, Canada ({\{\tt
      mpf,ggoh\}@cs.ubc.ca}). The work of these authors was supported
    by NSERC Discovery Grant 312104 and by NSERC CRD Grant
    375142.\hfill January 7, 2014}%
  \and Gabriel Goh$^*$} 
\begin{document}
\maketitle

\begin{abstract}
  Stochastic-approximation gradient methods are attractive for
  large-scale convex optimization because they offer inexpensive
  iterations. They are especially popular in data-fitting and
  machine-learning applications where the data arrives in a continuous
  stream, or it is necessary to minimize large sums of functions. It
  is known that by appropriately decreasing the variance of the error
  at each iteration, the expected rate of convergence matches that of
  the underlying deterministic gradient method. Conditions are given
  under which this happens with overwhelming probability.
\end{abstract}
\begin{keywords}
  stochastic approximation, sample-average approximation, incremental
  gradient, steepest descent, convex optimization
\end{keywords}
\begin{AMS}
  90C15, 90C25
\end{AMS}
\thispagestyle{plain}

\section{Introduction}
Stochastic-approximation methods for convex optimization are prized
for their inexpensive iterations and applicability to large-scale
problems. The convergence analyses for these methods typically rely on
expectation-based metrics for gauging progress towards a solution. But
because the solution path is itself stochastic,
practitioners---especially those relying on ad-hoc applications of
such algorithms for a limited number of iterations---may pause and
question how far an observed solution path is from the optimal
value. Our aim is to develop bounds on the probability of deviating
too far from the deterministic solution path. This result complements
existing expectation-based analyses, and can furnish useful guidance
for practitioners.

Let $f:\Real^{n}\to\Real$ be a differentiable function with a
Lipschitz-continuous gradient, and $g:\Real^{n}\to(-\infty,+\infty]$
be a convex lower-semicontinuous~\cite[\S7]{Roc70} function. Consider
the optimization problem
\begin{equation} \label{eq:7}
  \minimize{x\in\Real^{n}} \quad h(x) := f(x) + g(x).
\end{equation}
This formulation permits us to capture a wide variety of problems,
including convex constraints (by letting $g$ represent the indicator
function over that set) and nonsmooth regularizers.  We are interested
in the probabilistic guarantees of the approximate proximal-gradient
iteration
\begin{equation}
  \label{eq:27}
  x\kp1 = \prox_{\alpha\k}\!\!\set{x\k - \alpha\k (\nabla f(x\k) + e\k)},
\end{equation}
where  $\alpha\k$ is a 
positive step length, $e\k$ is a random variable, and
\[
  \prox_{\alpha}(z) := \argmin_{y}\set{\alpha g(y) +
    \tfrac{1}{2}\|x-y\|_{2}^{2}}
\]
is the proximal operator~\cite{ComWaj2005}.  The gradient residual
$e\k$ is meant to account for error that might be incurred in the
computation of the gradient $\nabla f(x\k)$.  Such situations may
arise, for example, if evaluating the exact gradient requires a costly
simulation, or traversing a large data set. This iteration reduces to
the classical steepest-descent method when $g\equiv0$ and
$e\k\equiv0$.

An application of this framework is to provide tail bounds for solving
stochastic optimization problems, such as when
\begin{equation*}
f(x) = \Ex F(x,Z),
\end{equation*}
where $Z$ is a random variable. Stochastic-approximation algorithms
generally proceed by generating at each iteration $k$ a random sample
$\{Z_{1},\dots,Z_{m_{k}}\}$ of size $m_{k}$, which is used to compute
the search direction
\begin{equation}
  \label{eq:28A}
  \nabla f(x_{k}) + e_{k}
  = \frac{1}{m_{k}}\sum_{i=1}^{m_{k}}\nabla_{x} F(x_{k},Z_{i}).
\end{equation}
When $Z$ takes on a finite number of values with uniform probability,
then $f$ is equivalent to the familiar case of sums of functions
\begin{equation} \label{eq:29}
f(x)=\frac{1}{M}\sum_{i=1}^{M}f_{i}(x),
\end{equation}
and then~\eqref{eq:28A} reduces to
\begin{equation}\label{eq:28}
\nabla f(x_{k}) + e_{k} =\frac{1}{m_{k}}\sum_{i\in\mathcal{S}_{k}}\nabla f_{i}(x_{k}) ,
\end{equation}
where the random sample $\mathcal{S}_{k}\subseteq \{ {1,\dots,M} \}$
is possibly chosen without replacement.

At one extreme is a fixed sample size $m\k$ (equal to 1, say), which
yields an inexpensive iteration but generally does not converge to a
minimizer unless $\alpha\k\to0$; at best it converges sublinearly to
the solution. At the other extreme is the deterministic
proximal-gradient method, which under certain conditions (described in
\S\ref{se:existing convergence}) converges linearly, i.e., for all
iterations $k$, there exists a constant $\rho<1$ such that
\begin{equation}\label{eq:3}
  \pi\k \le \rho^{k}\pi_{0},
\end{equation}
where $$\pi_{k}:=h(x_{k}) - \inf h(x)$$ is the gap between the current
and optimal values of the function.  As do Friedlander and
Schmidt~\cite{FriS:2011}, Byrd, Chin, Nocedal, and
Wu~\cite{Byrd:2012}, and So~\cite{So:2013}, we consider methods that
interpolate between these extremes by increasing the sample size at a
linear rate. As argued by Byrd et al.~\cite[see Table~1]{Byrd:2012},
the increasing-sample size strategy has a better complexity rate, in 
terms of total gradients sampled, than if the sample size is held fixed.

Our aim in this paper is to bound the probability that the rate of
convergence of the stochastic method deviates from 
linear-convergence rate, i.e., we provide tail bounds on
\begin{equation}
  \label{eq:30}
  \Pr\,(\pi_k-\rho^{k}\pi_{0}\ge \epsilon).
\end{equation}
It is straightforward to recast these results to obtain tail bounds on
$\Pr(\pi\k>\epsilon)$.  In \S\ref{sec:gradient-with-error} we describe
bounds that depend on the errors generically, and in
\S\ref{sec:tail-bounds-sagd} apply these results to obtain
exponentially decaying tail bounds in the case where the errors
decrease linearly. In \S\ref{sa} these results are further specialized
to the case where $f$ is given by~\eqref{eq:28A} and~\eqref{eq:28},
and exponential tail bounds are derived that depend on the sample
size.

\subsection{Assumptions and notation}
\label{sec:assumptions}
We make the following blanket assumptions throughout. The solution set
\Soln of \eqref{eq:7} is nonempty. For all $x$ and $y$, there exist 
positive constants $L$ and $\tau\geq 1$ such that
\begin{subequations}
\begin{align}
  \norm{\nabla f(y)-\nabla f(x)} 
  &\le L\,\norm{y-x}, 
  \label{a:lipschitz}
  \\ \min_{\xbar\in\Soln}\,\norm{x-\xbar}
  &\leq\tau\,\norm{x-\prox_{1/L}\set{x-\tfrac1L\nabla f(x)}}.
   \label{a:natural-res}
\end{align}
\end{subequations}
Except for the discussion in section~\ref{se:existing convergence},
$\alpha_{k}\equiv 1/L$.  Assumption~\eqref{a:lipschitz} asserts the
Lipschitz continuity of the gradient of
$f$. Assumption~\eqref{a:natural-res} is a global error bound on the
distance from $x$ to the solution set in terms of the residual in the
optimality conditions. Local versions of this error bound are
described by Tseng and Yun~\cite{Tsy09a} and Luo and
Tseng~\cite{LuT94}; the bound that we use here is a global version
described by So~\cite{So:2013}. This assumption is less restrictive
than strong convexity; in particular, if $g\equiv0$ and $f$ is
strongly convex with parameter $\mu$, \eqref{a:natural-res} holds with
$\tau = L/\mu$. Moreover, this assumption holds whenever $g$ is
polyhedral and one of the following holds for the function $f$: it is
convex and quadratic; or $f(x) = q(Ex) + c\T x$ for any matrix $E$,
vector $c$, and strongly convex function $q$; or $f(x)=\max_{\gamma\in
  Y}\,\set{\innerp{Ex}{y}-q(y)}$ for any strongly convex function $q$
that has a Lipschitz-continuous gradient. This is described in
Theorem~4 of Luo and Tseng~\cite{LuT94}, and Proposition~4 of
So~\cite{So:2013}.

Let $R\k:=\sum_{i=0}^{k-1}\rho^i$, where $\rho<1$ is a constant
specified in Lemma~\ref{lem:obj-val-bnd} in terms of $L$ and
$\tau$. Let $\Fscr\k=\sigma(e_1,e_2,\ldots,e\k)$ be the
$\sigma$-algebra generated by the sequence of errors $e_{i}$. When the context is clear, $[z]_{i}$ denotes the $i$th component of a
vector~$z$.

\subsection{Existing convergence analysis}
\label{se:existing convergence}

In general, if $\lim\inf_{k} \|e_{k}\| \neq 0$, then we necessarily
require $\alpha\k\to0$ in~\eqref{eq:27} in order to ensure optimality
of limit points. Combettes and Wajs~\cite[Theorem~3.4]{ComWaj2005}
show that the iteration~\eqref{eq:27} converges to a solution when
$0<\inf\alpha\k<\sup\alpha\k<2/L$ and
$\sum_{k\in\Natural}\norm{e\k}<\infty$, and also consider other kinds
of perturbations that enter into the iteration; no convergence rates
are given.  Schmidt, Roux, and Bach~\cite{SchmidtRouxBach:2011} link
the convergence rate of the iterations (including accelerated
variants) to $\Ex \|e_{k}\|^{2}$, which measures the variance in the
error, and to error in the proximal-map computation. In the case in
which $e\k=0$ has zero mean and finite variance, it is known that the
proximal-gradient method converges as $\BigOh(1/\sqrt k)$; see, e.g.,
Lanford, Li, and
Zhang~\cite{Langford:2009:SOL:1577069.1577097}. Contrast these rates
to those that can be obtained when $e\k\equiv0$, and in that case the
method has a rate of $\BigOh(1/k)$, and its accelerated variant has a
rate of $\BigOh(1/k^2)$, which is an optimal rate; see, e.g.,
Nesterov~\cite{Nes07} and Beck and Teboulle~\cite{BeT08}.

For the case where $g\equiv0$ (and hence the proximal operator
in~\eqref{eq:27} is simply the identity map), has been extensively
studied. It is well known that if $f$ is strongly convex,
deterministic steepest descent without error (i.e., $e\k\equiv0$) and
with a constant stepsize $\alpha\k = 1/L$ converges linearly with a
rate constant $\rho<1$ that depends on the condition number of $f$;
see~\cite[section~8.6]{luenberger2008linear}. Bertsekas and
Tsitsiklis~\cite{BT:2000} describe conditions for convergence of the
iteration~\eqref{eq:27} when the steplengths $\alpha\k$ satisfy the
conditions $\sum_{k\in\Natural} \alpha\k = \infty$ and
$\sum_{k\in\Natural} \alpha\k^2 < \infty$.  Bertsekas and
Nedi\'c~\cite{NedBer:2000} show that randomized incremental-gradient
methods for~\eqref{eq:29}, with constant steplength
$\alpha\k\equiv\alpha$, converge as 
\[ \Ex\pi\k \le \BigOh(1)(\rho^k +
\alpha) \] 
where $\BigOh(1)$ is a positive constant. This expression is
telling  because the first term on the right-hand side decreases at a
linear rate, and depends on the condition number through $\rho$; this
term is also present for any deterministic first-order method with
constant stepsize. For a constant stepsize, Friedlander and
Schmidt~\cite{FriS:2011} give non-asymptotic rates that directly
depend on the rate at which the error goes to zero, and for the case
where $f$ is given by~\eqref{eq:29}, they further showthe dependence
of the convergence rate on the sample size.

For a non-vanishing stepsize, i.e., $\liminf_{k}\alpha\k>0$, Luo and
Tseng~\cite{luo1993error} show that for a decreasing error sequence
that satisfies $\|e_{k}\| = \BigOh ( \|x_{k+1}-x_{k}\|)$, the function
values converge to the optimal value at an asymptotic linear rate.

The convergence in probability of the stochastic-approximation method
was first discussed by the classic 
Robbins~\cite{robbins1951stochastic} paper. Bertsekas and
Tsitsiklis~\cite{BT:2000} give mild conditions on $e_{k}$ and $f$
under which $f(x_{k}) \to \inf f(x)$ in probability. More recently,
Nemirovski, Juditsky, Lan, and Shapiro~\cite{nemirovski2009robust}
show that for decreasing steplengths $\alpha\k = \BigOh(1/k)$, these
methods achieve a sublinear rate according to $\Ex\pi\k =
\BigOh(1/k)$; the iteration average has similar convergence
properties, and it converges sublinearly with overwhelming
probability.

\section{Proximal point with error}

Our point of departure is the following result, which relates the
progress in the objective value to the norm of the gradient
residual. 

\begin{blemma}
  \label{lem:obj-val-bnd}
  After $k$ iterations of algorithm~\eqref{eq:27},
  \begin{equation*}
    \pi\k - \rho^k\pi_0 \le \frac{1}{\vartheta}
    \sum_{i=0}^{k-1}\rho^{k-1-i}\error{i},
  \end{equation*}
  where
  $\rho = (40\tau^{2})/(1+40\tau^{2})\in(0,1)$
  and
  $\vartheta = L\, (1+40\tau^{2})/(40\tau^{2})>0.$
\end{blemma}

The proof of this result is laid out in
Appendix~\ref{app:proof-of-obj-val-bnd}, and follows the template laid
out by Luo and Tseng \cite[Theorem~3.1]{luo1993error}, modified to
keep the error term $e\k$ explicit. So~\cite{So:2013} also provides a
similar derivation for the case where $g\equiv0$, in which case it
seems possible to obtain tighter constants $\rho$ and $\vartheta$.  If
additionally $\|e_{k}\|=0$, then the result reduces to the well-known
fact that steepest descent decreases the objective value linearly. We
note that the constants are invariant to scalings of $h$.

\begin{example}[Gradient descent with independent Gaussian noise, part
  I] \label{ex:expected-gaussian-noise} Let $e\k\sim N(0,\sigma^2
  I)$. Because $\error{k}$ is a sum of $n$ independent Gaussians, it
  follows a chi-squared distribution with mean
  $\Ex\error{k}=n\sigma^2$. Therefore,
  \begin{equation}
    \label{eq:4}
    \Ex\pi\k-\rho^k\pi_0
    \le \frac{1}{\vartheta} \sum_{i=0}^{k-1}\rho^{k-1-i}\Ex\error{i}
    = \frac{n\sigma^2}{\vartheta} \sum_{i=0}^{k-1}\rho^{k-1-i}.
  \end{equation}
  Take the limit inferior of both sides of~\eqref{eq:4}, and note that
  $\lim_{k\to\infty}\sum_{i=0}^{k-1}\rho^{k-1-i}=1/(1-\rho) $.  Use
  the values of the constants in Lemma~\ref{lem:obj-val-bnd} to obtain
  the bound
  \[
  \Ex\liminf_{k\to\infty}\pi\k
  \le\liminf_{k\to\infty}\Ex\pi\k
  \le(20\tau^2/L)\,n\sigma^2,
  \]
  where the first inequality follows from the application of Fatou's
  Lemma~\cite[Ch.~4]{RoydenF:2010}.  Hence, even though
  $\lim_{k\to\infty}\pi\k$ may not exist, we can still provide a lower
  bound on the distance to optimality that is proportional to the
  variance of the error term.
\end{example}

The following result establishes sufficient conditions under which the
distance to the solution $\pi\k$ exhibits a supermartingale
property. The dependence on the $\sigma$-algebra $\Fscr_{k-1}$ is
effectively a conditioning on the history of the algorithm.

\begin{btheorem}
  \label{thm:supermartingale} (Supermartingale Property).
  Let $\xbar\k$ be the projection of $x\k$ onto the solution set
  $\Soln$. For algorithm~\eqref{eq:27},
\[
\Ex [ \pi_{k+1} \mid \Fscr_{k-1}] \leq \pi_k 
\textt{if}
\Ex [ \|e_{k}\|^{2} \mid \Fscr_{k-1}]
\leq 1/(10 \tau^{2})\|x_{k}-\bar{x}_{k}\|^{2}.
\]
\end{btheorem}

\begin{proof} Lemma~\ref{lem:sufficient_decrease} gives a sufficient
  condition for the monotonicity of the iteration. Using that as a
  starting point yields
\begin{align*}
\pi_{k+1} & \leq\pi_{k}+\frac{1}{L}\|e_{k}\|^{2}-\frac{L}{4}\|x_{k}-x_{k+1}\|^{2}\\& 
 \overset{(i)}{\leq}\pi_{k}+\frac{1}{L}
  \|e_{k}\|^{2}-\frac{L}{4}
 \left(\frac{1}{2\tau^{2}}
  \|x_{k}-\bar{x}_{k}\|^{2}-\frac{5}{8L^{2}}\|e_{k}\|^{2}\right)\\& 
 \leq\pi_{k}+\frac{27}{32L}\|e_{k}\|^{2}-
  \frac{L}{8\tau^{2}}\|x_{k}-\bar{x}_{k}\|^{2},
\end{align*}
where $(i)$ comes from Lemma~\ref{lem:fourbounds}b. Taking
conditional expectations on both sides:
\begin{align*}
\Ex[\pi_{k+1} \mid \Fscr_{k-1}] & \leq 
  \Ex \Big[ \pi_{k}+\frac{27}{32L}\|e_{k}\|^{2}-\frac{1}{8\tau^{2}L}\|x_{k}-\bar{x}_{k}\|^{2} \mid \Fscr_{k-1} \Big] \\&
  \leq \pi_{k}+ \frac{27}{32L}\Ex [\|e_{k}\|^{2} \mid \Fscr_{k-1}] -\frac{1}{8\tau^{2}L}\|x_{k}-\bar{x}_{k}\|^{2} \leq \pi\k.
\end{align*}
\end{proof}

\section{Probabilistic bounds for gradient descent with random error}
\label{sec:gradient-with-error}

An immediate consequence of Lemma~\ref{lem:obj-val-bnd} is a tail
bound via Markov's inequality:
\[
\Pr(\pi\k-\rho^k\pi_{0} \ge \epsilon) \leq
\Pr \Big( \frac{1}{\vartheta} \sum_{i=0}^{k-1}\rho^{k-1-i}\error{i} \ge
\epsilon \Big)
\le  \frac{1}{\vartheta\epsilon}\sum_{i=0}^{k-1}\rho^{k-1-i}\Ex\error{i}.
\]
This inequality is too weak, however, to say anything meaningful about
the confidence in our solution after a finite number of iterations. We
are instead interested in Chernoff-type bounds that are exponentially
decreasing in $\epsilon$, and in the parameters that control the size
of the error.

The first bound (section~\ref{sec:gener-error-sequ}) that we develop
makes no assumption on the relation of the gradient errors between
iterations, i.e., the error sequence may or may not be history
dependent, and we thus refer to this as a generic error sequence. The
second bound (section~\ref{ubes}) makes the stronger assumption about
the relationship of the errors between iterations.

\subsection{Generic error sequence}\label{sec:gener-error-sequ}

Our first exponential tail bounds are defined in terms of the
moment-generating function
\begin{equation*}
\mgf\k(\theta) := \Ex\exp(\theta\error{k})
\end{equation*}
of the error norms $\error{k}$. We make the convention that
$\gamma\k(\theta)=+\infty$ for $\theta\notin\dom\gamma\k$.

\begin{btheorem}[Tail bound for generic errors]
  \label{th: generic tail bound} 
  For algorithm~\eqref{eq:27},
  \begin{subequations}
  \begin{equation}
    \label{eq:6}
    \Pr(\pi\k - \rho^k\pi_0 \ge \epsilon)
    \le
    \inf_{\theta>0}
    \left\{
      \frac{\exp(-\theta \vartheta \epsilon/R\k)}{R\k}
      \sum_{i=0}^{k-1}\rho^{k-1-i}\mgf_i(\theta)
    \right\}.
  \end{equation}
  If $\gamma\k\equiv\gamma$ for all $k$ (i.e., the error norms
  $\error{k}$ are identically distributed), then the bound simplifies
  to
  \begin{equation}
    \label{eq:6-simple}
    \Pr(\pi\k - \rho^k\pi_0 \ge \epsilon)
    \le
    \inf_{\theta>0}
    \left\{
    	\exp(-\theta \vartheta \epsilon/ R\k) 
      \mgf(\theta)
    \right\}.
  \end{equation}
  \end{subequations}
\end{btheorem}
\begin{proof}
  By the definition of $R_{k}$, $\left(\sum_{i=0}^{k-1}
    \rho^{k-1-i}\right)/{R\k}=1$. Thus, for $\theta>0$,
  \begin{align*}
    \Ex
        \exp\left(
              \theta\sum_{i=0}^{k-1}\rho^{k-1-i}\error{i}
            \right)
     &=
    \Ex
        \exp\left(
              \sum_{i=0}^{k-1}
              \frac{\rho^{k-1-i}}{R\k}\theta R\k\error{i}
            \right)
\\  &\overset{(i)}\le
    \Ex
        \sum_{i=0}^{k-1}\frac{\rho^{k-1-i}}{R\k}
        \exp(\theta R\k\error{i})
\\  &\overset{(ii)}=
    \frac{1}{R\k}\sum_{i=0}^{k-1}\rho^{k-1-i}\mgf_i(\theta R\k),
  \end{align*}
  where $(i)$ follows from the convexity of $\exp(\cdot)$, and $(ii)$
  follows from the linearity of the expectation operator and the
  definition of $\mgf_i$. Together with Markov's inequality, the above
  implies that for all $\theta>0$,
  \begin{align}
    \Pr\left(\sum_{i=0}^{k-1}\rho^{k-1-i}\error{i}\ge\epsilon \right)
    &=
    \Pr\left(
      \exp
      \left[
        \theta \sum_{i=0}^{k-1}\rho^{k-1-i}\error{i}
      \right]
      \ge
      \exp(\theta\epsilon)
    \right) \nonumber
\\ &\le
   \exp(-\theta\epsilon)
   \Ex\exp\left(
     \theta \sum_{i=0}^{k-1}\rho^{k-1-i}\error{i}
   \right) \nonumber
\\ &\le
   \frac{\exp(-\theta\epsilon)}{R\k}
   \sum_{i=0}^{k-1}\rho^{k-1-i}\mgf_i(\theta R\k).
   \label{eq:markov-application-sum}
  \end{align}
  This inequality, together with Lemma~\ref{lem:obj-val-bnd}, implies
  that for all $\theta>0$,
  \begin{align*}
    \Pr\left(\pi\k-\rho^k\pi_0\ge\epsilon\right)
    &\le
    \Pr
    \left(
      \frac1{\vartheta}\sum_{i=0}^{k-1}\rho^{k-1-i}\error{i}
      \ge \vartheta\epsilon
    \right) 
\\  &\le
    \frac{\exp(-\theta \vartheta\epsilon)}{R\k}
    \sum_{i=0}^{k-1}\rho^{k-1-i}\mgf_i(\theta R\k),
  \end{align*}
  where we use the elementary fact that
  $\Pr(X\ge\epsilon)\le\Pr(Y\ge\epsilon)$ if $X\le Y$ almost
  surely. Redefine $\theta$ as $\theta R\k$, and take the infimum
  of the right-hand side over $\theta>0$, which gives the required
  inequality~\eqref{eq:6}. The simplified bound~\eqref{eq:6-simple}
  follows directly from the definition of $R\k$.
\end{proof}

When the errors are identically distributed, there is an intriguing
connection between the tail bounds described in Theorem~\ref{th:
  generic tail bound} and the convex conjugate of the
cumulant-generating function of that distribution, i.e., $(\log\circ
\,\gamma)^{*}$.

\begin{bcorollary}[Tail bound for identically-distributed errors]
  \label{co:simple-log}
  Suppose that the error norms $\error{k}$ are identically
  distributed. Then for algorithm~\eqref{eq:27},
  \begin{equation*}
    \log\Pr(\pi\k - \rho^k\pi_0 \ge \epsilon)
    \le
    -\left[\log\mgf(\cdot)\right]^*(\vartheta\epsilon/R\k).
  \end{equation*}
\end{bcorollary}
\begin{proof}
  Take the log of both sides of~\eqref{eq:6-simple} to get
  \begin{align*}
    \log\Pr(\pi\k - \rho^k\pi_0 \ge \epsilon)
     &\le
     \log\inf_{\theta>0}
     \left\{
       \exp(-\theta \vartheta\epsilon/R\k)\, \mgf(\theta)
     \right\}
   \\&=-\sup_{\theta>0}
   \left\{
     (\vartheta\epsilon/R\k)\theta - \log\mgf(\theta)
   \right\},
  \end{align*}
  which we recognize as the negative of the conjugate of
  $\log\circ\,\mgf$ evaluated at $\vartheta\epsilon/R\k$.
\end{proof}

Note that these bounds are invariant with regard to scaling, in
the sense that if the objective function $f$ is scaled by some
$\alpha>0$, then the bounds hold for $\alpha\epsilon$.

The following example illustrates an application of this tail bound to
the case in which the errors follow a simple distribution with a known
moment-generating function.

\begin{example}[Gradient descent with independent Gaussian noise,
  part~II]
  \label{ex:expected-gaussian-noise-2}
  As in Example~\ref{ex:expected-gaussian-noise}, let
  $e\k\sim N(0,\sigma^2 I)$. Then $\error{k}$ is a scaled chi-squared
  distribution with moment-generating function
  \[
  \mgf\k(\theta) = (1-2\sigma^2 \theta)^{-n/2},
  \quad
  \theta \in \left[0,\frac1{2\sigma^2}\right).
  \]
  Note that
  \[
  [\log\mgf(\cdot)]^*(\mu)
  = \frac{\mu-n\sigma^{2}}{2\sigma^{2}}+\frac{n}{2}\log(n\sigma^{2}/\mu)
  \text{for}
  \mu > n\sigma^{2}.
  \]
  We can then apply Corollary~\ref{co:simple-log} to this case to
  deduce the bound
  \[
  \Pr(\pi\k-\rho^k\pi_0\ge \epsilon)
  \le
  \left(\frac{\exp(1)}{n}\cdot\frac{\vartheta\epsilon}{\sigma^{2} R\k}\right)^{n/2}\!\!
  \exp\left(-\frac{\vartheta\epsilon}{2 \sigma^{2} R\k}\right)
  \text{for}
  \epsilon>\frac{n\sigma^{2} R_k}{\vartheta}.
  \]
  The bound can be further simplified by introducing an additional
  perturbation $\delta>0$ that increases the base of the exponent:
  \begin{equation}\label{eq:13}
  \Pr(\pi\k-\rho^k\pi_0\ge \epsilon) =
  \Oscr\left[\exp\left(-\delta\frac{\vartheta\epsilon}{2 \sigma^{2}
        R\k}\right)\right] \text{for all} \mbox{$\delta\in[0,1)$},
  \end{equation}
  which highlights the exponential decrease of the bound in terms of
  $\epsilon$.

\end{example}

\subsection{Unconditionally bounded error sequence} \label{ubes}

In contrast to the previous section, we now assume that there exists a
deterministic bound on the conditional expectation
$\Ex\left[\exp(\theta\error{k})\mid\Fscr\km1\right]$. We say that this
bound holds unconditionally because it holds irrespective of the
history of the error sequence.

\begin{bassumption} \label{as:unconditional errors} Assume that
  $\Ex\left[\exp(\theta\error{k})\mid\Fscr\km1\right]$ is finite over
  $[0,\sigma)$, for some $\sigma>0$. Therefore there exists, for each
  $k$, a deterministic function
  $\bmgf\k:\Real_{+}\to\Real_{+}\cup\{\infty\}$ such that
  \[
  \bmgf\k(0)=1
  \text{and}
  \Ex\left[\exp(\theta\error{k})\mid\Fscr\km1\right]\le\bmgf\k(\theta).
  \]
  (Thus, the bound is tight at $\theta=0$.)
\end{bassumption}

The existence of such a function in fact implies a bound on the
moment-generating function of $\error{k}$. In particular,
\begin{equation} \label{eq:8}
\mgf\k(\theta) := 
\Ex \exp(\theta\error{k}) =
\Ex\left[
      \Ex\left[
        \exp(\theta\error{k})\mid\Fscr\km1
      \right]
    \right]
\le\Ex \bmgf\k(\theta) =
\bmgf\k(\theta).
\end{equation}
The converse, however, is not necessarily true. To see this, consider
the case in which the errors $e_1,\ldots,e\km1$ are independent
Bernoulli-distributed random variables, and $e\k$ is a deterministic
function of all the previous errors, e.g., $\Pr(e_i=0) = \Pr(e_i=1) =
1/2$ for $i=1,\ldots,k-1$, and the error on the last iteration is
completely determined by the previous errors:
\[
e\k =
\begin{cases}
  1 & \mbox{if $e_1 = e_2 = \cdots = e\km1$},
\\0 & \mbox{otherwise}.
\end{cases}
\]
Therefore, $\Pr(e\k=1) = (1/2)^{k-1}$ and $\Pr(e\k=0) =
1-(1/2)^{k-1}$, and the moment-generating function of $e\k$ is
$
\gamma\k(\theta) = 1-{2^{1-k}}(1+\exp\theta).
$
Then,
\[
\Ex[\exp(\theta e\k^2)\mid e_1,\ldots,e\km1] =
\begin{cases}
  \exp\theta   & \mbox{if $e_1 = e_2 = \cdots = e\km1$},
\\1            & \mbox{otherwise,}
\end{cases}
\]
whose tightest deterministic upper bound is
$\bmgf\k(\theta)=\exp\theta$. However,
$\bmgf\k(\theta)\ge\gamma\k(\theta)$ for all $\theta\ge0$.

The following result is analogous to Theorem~\ref{th: generic tail bound}.

\begin{btheorem}[Tail bounds for unconditionally bounded errors]
  \label{th:unconditional error bound}
  Suppose that Assumption~\ref{as:unconditional errors} holds. Then
  for algorithm~\eqref{eq:27},
  \begin{equation*}
    \Pr(\pi\k-\rho^k\pi_0\ge\epsilon)
    \le
    \inf_{\theta > 0}
    \left\{
      \exp(-\theta \vartheta\epsilon)
      \prod_{i=0}^{k-1}\bmgf_i(\theta\rho^{k-i-1})
    \right\}.
  \end{equation*}
\end{btheorem}
\begin{proof}
  The proof follows the same outline as many martingale-type
  inequalities \cite{azuma1967weighted,chung2006concentration}. We
  obtain the following relationships:
  \begin{align*}
    \Ex \exp\left[ \theta\sum_{i=0}^{k-1}\rho^{k-1-i}\error{i} \right]
    &\overset{(i)}{=}
    \Ex\left[
        \Ex\left[
            \exp\left.
                \left[
                 \theta\sum_{i=0}^{k-1}\rho^{k-1-i}\error{i}
                \right]
                \right|\Fscr\km2
           \right]
       \right]
\\  &=
    \Ex\left[
        \Ex\left[
            \exp\left.
                \left[
                 \theta\rho^0\error{k-1}+\theta\sum_{i=0}^{k-2}\rho^{k-1-i}\error{i}
                \right]
                \right|\Fscr\km2
           \right]
       \right]  
\\  &\overset{(ii)}{=}
    \Ex\!\left[
            \exp\left[
                 \theta\sum_{i=0}^{k-2}\rho^{k-1-i}\error{i}
                 \right]
            \Ex\left[
                  \left.
                  \exp\left(
                        \theta\error{k-1}
                       \right)
                  \right| \Fscr\km2
               \right]
       \right]
\\    &\overset{(iii)}{\leq}
      \Ex \left[
        \exp \left[
          \theta \sum_{i=0}^{k-2}\rho^{k-i-1}\error{i}
          \right]
        \right]
      \bmgf_{k-1}(\theta)
\\    &\overset{(iv)}{\leq}
      \prod_{i=0}^{k-1}\bmgf_i(\theta \rho^{k-i-1}),
\end{align*}
where $(i)$ follows from the law of total expectations, i.e.,
$\Ex_Y[\Ex[X|Y]]=\Ex[X]$; $(ii)$ follows from the observation that the
random variable $\exp(\theta\sum_{i=0}^{k-2}\rho^{k-1-i}\error{i})$ is
a deterministic function of $e_0,\ldots,e\km2$, and hence is
measurable with respect to $\Fscr\km1$ and can be factored out of the
expectation; $(iii)$ uses Assumption~\ref{as:unconditional errors};
and to obtain $(iv)$ we simply repeat the process recursively.

Thus, we now have a bound on the moment-generating function of the
discounted sum of errors
$\theta\sum_{i=0}^{k-1}\rho^{k-1-i}\error{i}$, and we can continue by
using the same approach used to
derive~\eqref{eq:markov-application-sum}. The remainder of the proof
follows that of Theorem~\ref{th: generic tail bound}, except that the
sums over $i=0,\ldots,k$ are replaced by products over that same
range.
\end{proof}

In an application where both $\mgf\k$ and $\bmgf\k$ are available, it
is not true in general that either of the bounds obtained in
Theorems~\ref{th: generic tail bound} and~\ref{th:unconditional error
  bound} are tighter than the other. When only a bound $\bmgf\k$ that
satisfies Assumption~\ref{as:unconditional errors} is available,
however, (which is the case in the sampling application of
section~\ref{sa}) we could leverage~\eqref{eq:8} and apply
Theorem~\ref{th: generic tail bound} to obtain a valid bound in terms
of $\bmgf\k$ by simply substituting it for $\mgf\k$. However, as shown
below, in this case it is better to apply
Theorem~\ref{th:unconditional error bound} because it yields a
uniformly better bound:
\begin{equation}\label{eq:9}
  \Pr\left(\pi_k-\rho^k\pi_k\geq\epsilon\right)
  \leq
  \inf_{\theta>0}
  \left\{
    \exp
    \left(
      -\theta \vartheta
      \epsilon+\sum_{i=0}^{k-1}\log\bmgf_{i}\big(\theta\rho^{k-1-i}\big)
    \right)
  \right\},
\end{equation}
while Theorem~\ref{th: generic tail bound} (with $\mgf\k$ replaced by
$\bmgf\k$) gives us
\begin{equation} \label{eq:10}
  \Pr\left(\pi\k-\rho^k\pi_{0}\geq\epsilon\right)
  \leq
  \inf_{\theta>0}
  \left\{
    \exp\left(-\theta \vartheta\epsilon
      +\log\left[
        \frac{1}{R_k}\sum_{i=0}^{k-1}\rho^{k-1-i}\bmgf_{i}(\theta R_k)
      \right]
    \right)
  \right\},
\end{equation}
where we rescale $\theta$ by $R\k$.  A direct comparison of the two
bounds show that they only differ by one term:
\[
  \log\left[
    \frac{1}{R_{k}}\sum_{i=0}^{k-1}\rho^{k-i-1}\bmgf_{i}(\theta R_k)
  \right]
  \textt{vs.}
  \sum_{i=0}^{k-1}\log\bmgf_{i}(\theta\rho^{k-1-i}).
\]
Because $R\k=\sum_{i=0}^{k}\rho^{k-i-1}$, the term in the
$\log$ on the left is a convex combination of the functions
$\bmgf_{i}$. Therefore,
\begin{align*}
  \log\left[
    \frac{1}{R_{k}}\sum_{i=0}^{k-1}\rho^{k-i-1}\bmgf_{i}(\theta R_k)
  \right]
    & \overset{(i)}{\ge}
  \sum_{i=0}^{k-1}\frac{\rho^{k-1-i}}{R_k}\log\bmgf_{i}(\theta R_k)
  \\& \overset{(ii)}{\ge}
  \sum_{i=0}^{k-1}\log\bmgf_{i}(\theta R_k\,\rho^{k-1-i}/R_k)
  \\& = 
  \sum_{i=0}^{k-1}\log\bmgf_{i}(\theta\rho^{k-1-i}),
\end{align*}
where $(i)$ is an application of Jensen's inequality and the concavity
of $\log$, and $(ii)$ follows from the convexity of the cumulant
generating function. It is then evident that~\eqref{eq:9}
implies~\eqref{eq:10}.

As with Corollary~\ref{co:simple-log}, by taking logs of both sides
above, a connection can be made between our bound and the infimal
convolution when $\gammabar$ is log-concave:
\begin{align*}
\log \Pr(\pi\k-\rho^{k}\pi_{0}\ge\epsilon) & \le 
  \left[\bigotimes_{i=0}^{k-1}\, [\log\bmgf_{i}(\,\cdot\,\rho^{k-i-1})]^{*}
  \right](\vartheta\epsilon/R_{k}),
\end{align*}
where $\otimes$ denotes the infimal convolution operator.

\begin{example}[Gradient descent with independent Gaussian noise, part
  III]\label{ex:ex-gaussian-noise-3}
  As in Example~\ref{ex:expected-gaussian-noise-2}, let $e\k\sim
  N(0,\sigma^2 I)$. Because the errors $e\k$ are independent,
  $\Ex\big[\exp(\theta\error{k})\,|\,\Fscr\km1\big]=\Ex\exp(\theta\error{k})=\mgf\k(\theta)$,
  which satisfies Assumption~\ref{as:unconditional errors} with
  $\bmgf\k(\theta):= \mgf\k(\theta)$. Apply
  Theorem~\ref{th:unconditional error bound} to obtain the bound
  \begin{align}\label{eq:12}
    \Pr \, ( 
    \pi_k-\rho^k\pi_{0}
      \geq\epsilon
      )
    &\leq
    \inf_{\theta>0}
    \left\{
      \exp(-\theta \vartheta \epsilon) \cdot 
      \prod_{i=0}^{k-1}(1-2\sigma^2\theta\rho^{k-1-i})^{-n/2}\right\}.
  \end{align}
Apply Lemma~\ref{fact:Q-Pochhammer_Lower_Bound} to obtain
\[
  \Pr\,
  (
    \pi_k-\rho^k\pi_{0}
      \geq\epsilon
  )
  \overset{}{\leq}
  \left(
    \frac{\exp(1)}{n\alpha}\cdot\frac{ \vartheta \epsilon }{\sigma^2}
  \right)
  ^{\tfrac{n\alpha}{2}} \!\!
  \exp
  \left(
    -\frac{\vartheta\epsilon}{\sigma^2}
  \right)
  \text{for} 
  \epsilon > \frac{n\alpha\sigma^2}{\vartheta},
\]
where $\alpha=1-(\log\rho)^{-1}.$ We simplify the bound to obtain
  \begin{equation}\label{eq:14}
  \Pr\,(\pi\k-\rho^k\pi_0\ge \epsilon) =
  \Oscr\left[\exp\left(- \,\delta\cdot \frac{\vartheta\epsilon}{\sigma^{2}}\right)\right] \text{for all} \mbox{$\delta\in(0,1)$};
  \end{equation}
cf.~\eqref{eq:13}.

As an aside, we note that we can easily accommodate correlated noise,
i.e., $\e\k\sim N(0,\Sigma^{2})$ where $\Sigma$ is an $n\times n$ positive
definite matrix. The error $\error{k}$ then has the
distribution of a sum of chi-squared random variables that are
weighted according to the eigenvalues $\sigma_j$ of
$\Sigma$~\cite{imhof:1961}:
  \[
  \error{k} \sim \sum_{j=1}^{n}\sigma_{j}^{2}\chi_1^2,
  \]
  and so the above tail bounds hold with $\sigma=\sigma\submax$.
\end{example}{

  The bounds obtained in Examples~\ref{ex:expected-gaussian-noise-2}
  and~\ref{ex:ex-gaussian-noise-3} illustrate the relative strengths
  of Theorems~\ref{th: generic tail bound} and~\ref{th:unconditional
    error bound}. Comparing~\eqref{eq:13} and~\eqref{eq:14}, we see
  that the asymptotic bounds only differ by a factor of
  $1/R_k$. Hence, for large $\epsilon$, the bound in
  Example~\ref{ex:expected-gaussian-noise-2} is uniformly weaker than
  the bound in Example~\ref{ex:ex-gaussian-noise-3}. Note that this
  holds despite the simplification (i.e.,
  Lemma~\ref{fact:Q-Pochhammer_Lower_Bound}) used to
  simply~\eqref{eq:12}.

\section{From tail bounds to moment-generating bounds}
\label{sec:from-tail-bounds}

Let $\mathcal{G}$ be a $\sigma$-algebra.  Consider the exponential
bound on the conditional probability~\cite[Definition
8.11]{klenke2008probability} of a sequence of univariate random
variables $X_{i}$:
\begin{equation}\label{eq:15}
  \Pr(X_{i}\geq \epsilon\mid\mathcal{G})
  :=  \Ex[\ind{X_{i}\ge\epsilon} \mid \mathcal{G}\,]
  \le \exp(-\epsilon^{2}/\nu)
  \quad\mbox{for some}\quad\nu>0.
\end{equation}
In this section we show that this bound translates into a
deterministic bound on the conditional moment-generating function
\[
\Ex [\exp(\theta\|X\|^{2}) \mid \mathcal{G}],
\]
where $X=(X_{1},X_{2},\dots,X_{n})$ is an $n$-vector.  The subsidiary
lemmas follow standard arguments (e.g., see
\cite[Chapter~2]{boucheron:2013}), except for the requirement to
condition on $\mathcal{G}$; hence, we rederive the required results.

\begin{blemma}[Bounds on moments] \label{le:bounds on moments}
  If~\eqref{eq:15} holds for some $\nu>0$, then 
  \[
  \Ex [X_{i}^{2v}\mid \mathcal{G}]\leq v!\nu^{v}
  \textt{for all}
  v=0,1,2,\ldots.
  \]
\end{blemma}
\begin{proof}
  We follow a similar argument to \cite[Theorem~2.1]{boucheron:2013}). Use the substitution $\epsilon^{2v}=\tau$ to obtain
   \[
   \Pr\big(Y^{2v}\geq \tau \mid \calG \big)\leq\exp\big({-\tau^{1/v}/\nu}\big).
   \]
  Integrate to get
   \[
   \Ex [Y^{2v}\mid\calG]
   =
   \int_{0}^{\infty}
   \!\!\Ex[
    \ind{
      Y{}^{2v}\geq \tau
    }\mid\mathcal{G}]
    \, d\tau\;\leq\;
    \int_{0}^{\infty}{\textstyle \!\!\exp
      (
        -\tau^{1/v}/\nu
      )
      \, d\tau\;=\;\Gamma(1+v)\nu^{v}\;=\; v!\nu^{v},}
   \]
   where the first equality comes from the conditional layer-cake
   representation of positive random
   variables~\cite{Swanson:2013:cond-layer-cake}.
\end{proof}

With this result, we can translate the bound~\eqref{eq:15} into a
bound on the moment-generating function of $Y^{2}$.

\begin{blemma}[Bound on conditional MGF]
  \label{lem:Relationship between exponential bound and MGF} If~\eqref{eq:15} holds for some $\nu > 0$, then
  \[
  \Ex[\exp\,(\theta Y^{2}) \mid \calG]\leq\frac{1}{1-\theta\nu}
  \textt{for} \theta \in [0,1/\nu).
  \]
\end{blemma}
\begin{proof}
  Using the Taylor expansion of $\Ex[\exp\big(\theta Y^{2}\big) \mid \calG]$,
  \begin{align*}
    \Ex[\exp\big(\theta Y^{2} \big) \mid \calG]
    & =\Ex \left[
             \left.
               \sum_{i=0}^{\infty}\theta^{i}\frac{Y^{2i}}{i!}
            \,\right|\, \calG
          \right]\\
    & \overset{(i)}{=}\sum_{i=0}^{\infty}\theta^{i}\frac{\Ex [Y^{2i}
      \mid \calG]}{i!} \\
    &
    \overset{(ii)}{\leq}\sum_{i=0}^{\infty}\theta^{i}\frac{i!\nu^{i}}{i!}
    =\sum_{i=0}^{\infty}(\theta\nu)^{i}=\frac{1}{1-\theta\nu}.
  \end{align*}
  Equality $(i)$ is obtained via the conditional monotone
  convergence theorem~\cite[Theorem~9.7e]{williams91:_probab_martin},
  which allows us to exchange limits and conditional expectations;
  inequality $(ii)$ is obtained using Lemma~\ref{le:bounds on
    moments}.
\end{proof}

We now generalize this last result to the case in which $X$ is a
random $n$-vector.

\begin{btheorem}[From tail bounds to moment-generating bounds]
  \label{le:Vector-mgf-bounds}
  Let $X$ be a random $n$-vector for which the tail
  bound~\eqref{eq:15} holds for each $i$ for some $\nu > 0$.  Then
  \[
  \Ex [ \exp(\theta\|X\|^{2}) \mid \calG]
  \leq\frac{1}{1-\theta \nu n}
  \textt{for} \theta \in [0, 1/\nu n).
  \]
\end{btheorem}

\begin{proof}
  From Lemma~\ref{lem:Relationship between exponential bound and MGF},
\begin{equation}\label{eq:17}
\Ex\left[ \exp \big( \theta n\big[X\big]_{i}^{2} \big)\mid \calG\right]
\le\
\frac{1}{1-\theta n \nu }.
\end{equation}
The following inequalities hold:
\begin{align*}
  \Ex\left[\exp\left.
      (\theta\|X\|^{2}) \right| \calG
      \right]
  &=\mathbf{E}\left[ \exp
    \left(\left.
      \theta\sum_{i=1}^{n}\big[X\big]{}_{i}^{2}
    \right)
    \right|\calG\right]
  \\
  &=\Ex\left[\exp\left.
  \left(
    \theta n \sum_{i=1}^{n}\frac{1}{n}\big[X\big]_{i}^{2}
  \right)
  \right|\calG \right]
  \\
  &\overset{(i)}{\leq}
  \Ex\left[\left.\sum_{i=1}^{n}\frac{1}{n}\exp
  \left(
    \theta n\big[X\big]_{i}^{2}
  \right)\right|\calG\right]
  \\
  &=
  \sum_{i=1}^{n}\frac{1}{n}
  \Ex\left[\exp\left.
  \left(
    \theta n\big[X\big]_{i}^{2}
  \right)\right|\calG\right]
  \overset{(ii)}{\leq}\frac{1}{1-\theta n \nu },
\end{align*}
where $(i)$ follows from Jensen's inequality and $(ii)$ follows
from~\eqref{eq:17}.
\end{proof}

\section{Convergence rates for linearly decreasing errors}
\label{sec:tail-bounds-sagd} 

Section~\ref{sec:gradient-with-error} describes tail bounds
for~\eqref{eq:30} in terms of any available bound on the
moment-generating function of the error $e\k$. A goal of this section
is to show that an exponential tail bound on the error translates to
an exponential tail bound on~\eqref{eq:30}. Thus we consider the case
where the tails on each component of $e\k$ are exponentially
decreasing (cf. Hypothesis~\ref{as:pop-bounds}.B below). We also
consider two additional conditions on the error sequence, which
illustrate the exponential tail bound's relative strength in the
following hierarchy of assumptions. In section~\ref{sa} we show how
various sampling strategies satisfy these conditions.

\begin{bhypothesis}[Uniform bounds] \label{as:pop-bounds}
  Suppose that for each $\beta\in(0,1)$,
  \begin{equation}\label{eq:31}
    U_{k} \leq \lambda\beta^{k}
  \end{equation}
  for some constant $\lambda>0$ and for all $k$. Consider the
  following hypotheses:

  \medskip
  \begin{tabular}{l@{\ }ll}
    A. &[Variance] &$\Ex \error{k} \leq U_{k}$;
    \\[6pt]
    B. &[Exponential Tail] & $
    \Pr\left([e_{k}]_{i}\geq\epsilon\mid\Fscr\km1\right)
    \leq\exp\big({-\epsilon^{2}/U_{k}}\big)$;
  \\[6pt]
  C. &[Norm] &$\error{k} \leq U_{k}$.
  \end{tabular}
\end{bhypothesis}
These conditions are ordered in increasing strength: if (C) holds,
then (B) holds by Hoeffding's inequality
(Theorem~\ref{thm:Hoeffding Bound}), and if (B) holds, then
(A) holds because the exponential bound implies a bound on the second
moment, i.e.,
\[
\Ex \big[[e_{k}]_{i}^{2}\mid \Fscr\km1\big] = \int_{0}^{\infty}\!\! 
\Pr([e_{k}]_{i}^{2} \geq \epsilon\mid\Fscr\km1)\,d\epsilon \leq \int_{0}^{\infty}\!\!
\exp\big({-\epsilon^{2}/U_{k}}\big)\,d\epsilon < \infty.
\]

\subsection{Expectation-based and deterministic bounds} Although our
main goal is to derive tail bounds, it is useful to compare these
against the expectation-based and deterministic bounds derived in
Friedlander and Schmidt~\cite[Theorem 3.3]{FriS:2011}. We give here a
reformulation of these results, which rely on parts A and C of
Hypothesis~\ref{as:pop-bounds}.

\begin{btheorem} [Bound in
  expectation] \label{th:bound-in-expectation} If
  Hypothesis~\ref{as:pop-bounds}.A holds, then
  \[
  \Ex\pi_{k}-\rho^{k}\pi_{0} = \Oscr ([\max\{\beta,\rho\}+\zeta]^{k})
  \text{for all} \zeta>0,
  \]
  and if $\rho\ne\beta$, then the bound holds with $\zeta = 0$. If
  Hypothesis~\ref{as:pop-bounds}.C holds, than this result holds
  verbatim, except without the expectation operator.
\end{btheorem}

\begin{proof}
  For $\beta \leq \rho $, it follows from Lemma~\ref{lem:obj-val-bnd}
  and Hypothesis~\ref{as:pop-bounds}.A that
  \begin{equation} \label{eq:rho-ne-beta}
    \Ex \pi_{k}-\rho^{k}\pi_{0}  
      \leq\frac{1}{\vartheta}{\sum_{i=0}^{k-1}}\rho^{k-i-1}\Ex \error{i}
    \leq
    \frac{\lambda\rho^{k-1}}{\vartheta}\sum_{i=0}^{k-1}(\beta/\rho)^{i}
    \leq
    \frac{\lambda}{\vartheta}\rho^{k-1}k.
  \end{equation}
  Similarly, for $\beta > \rho$, 
  \begin{equation} \label{eq:rho > beta}
  \Ex \pi_{k}-\rho^{k}\pi_{0}
  \le \frac{\lambda\beta^{k-1}}{\vartheta}\sum_{i=0}^{k-1}(\rho/\beta)^{i}
  \leq\frac{\lambda}{\vartheta}\,\beta^{k-1}k.
  \end{equation}
  We summarize these last two bounds in the single expression
  \begin{equation*}
  \Ex \pi_{k}-\rho^{k}\pi_{0}\leq\frac{\lambda}{\vartheta}\max\left\{
    \beta,\rho\right\} ^{k-1}k = \Oscr( [\max\left\{
    \beta,\rho\right\}+\zeta]^{k} ) 
  \end{equation*}
  for all $\zeta > 0$.
  
  If $\beta\ne\rho$, then it follows from the second inequality
  in~\eqref{eq:rho-ne-beta} and the first inequality in~\eqref{eq:rho
    > beta}, and the summation formula for geometric series, that
  \begin{equation}
    \label{na bound}
  \Ex\pi_{k} - \rho^{k}\pi_{0} \le \frac{\lambda}{\vartheta}
  \max\{\beta,\rho\}^{k-1}\frac1{|\beta-\rho|}
  =\Oscr(\max\{\beta,\rho\}^{k}).
  \end{equation} 

  If Hypothesis~\ref{as:pop-bounds}.C holds, than the proof above
  proceeds verbatim, except that the expectation operator above can be
  removed.
\end{proof}

\subsection{Tail bounds}

The next result gives exponential tail bounds in terms the iteration
$k$, and the deviation $\epsilon$ from the linear rate of
deterministic steepest descent.

\begin{btheorem}[Tail bounds]\label{th:tail bounds} If
  Hypothesis~\ref{as:pop-bounds}.B holds, then
\begin{equation}\label{eq:1}
 \Pr(\pi\k-\rho^{k}\pi_{0}\ge\epsilon)=\Oscr
 \left(
   \exp\left[-\frac{\epsilon}{\max\{\beta,\rho\}^{k}}\cdot \zeta\right]
 \right)
 \text{for some} \zeta>0.
\end{equation}
 \end{btheorem}
\begin{proof}
  From Theorem~\ref{le:Vector-mgf-bounds} the conditioned
  moment-generating function of $\error{k}$ is bounded:
  \begin{equation}\label{eq:2}
  \Ex[\exp(\theta\error{k})\mid\Fscr_{k-1}]
  \le
  \frac{1}{1-\theta n U_{k}}
  \text{for}
  \theta \in \left[0, \frac1{nU_{k}} \right).
  \end{equation}
  Define
  $$
  \alpha_1 = \max_{k} \rho^{k-i-1}  nU_{k}
  \text{and}
  \alpha_2 =\max\left\{ \beta,\rho\right\}.
  $$
  We can now use Theorem~\ref{th:unconditional error bound}, where we
  identify $\gammabar$ with the bound in~\eqref{eq:2} (and define
  $\gammabar(\theta) = \infty$ outside of the required interval), to
  obtain the tail bound
  \begin{align}
  \Pr(\pi_{k}-\rho^{k}\pi_{0}\geq\epsilon)
  &\overset{(i)}\leq \inf_{\theta\in[0,1/\alpha_1)}
        \left\{ \frac{\exp(-\theta \vartheta \epsilon)}
          {\prod_{i=0}^{k-1}\big(1-\theta n U_{k}\rho^{k-i-1}\big)}
        \right\}
   \nonumber
\\&\overset{(ii)}{\leq}\inf_{\theta\in[0,1/\alpha_1)}
  \left\{ \frac{\exp(-\theta\vartheta\epsilon)}{\prod_{i=0}^{k-1}
      (1-\theta  n\lambda \beta^{i}\rho^{k-i-1})}
  \right\}
  \nonumber
\\&\overset{(iii)}{=}\inf_{\theta\in[0,1/\alpha_1)}
  \left\{ \frac{\exp(-\theta\vartheta\epsilon )}{
      \prod_{i=0}^{k-1}(1-\theta  n\lambda \alpha_2^{k-1}\min\{\beta/\rho,\,\rho/\beta\}^{i})
    }
  \right\}, \label{eq:20}
\end{align}
where $(i)$ follows from the definition of $\alpha_1$, $(ii)$ follows from~\eqref{eq:31}, and $(iii)$ follows from
the definition of $\alpha_2$.

Define
$\alpha_3=1+1/\log(1/\min\{\beta/\rho,\,\rho/\beta\})=1+1/\left|\log\beta-\log\rho\right|$,
and apply Lemma~\ref{Fact:Tedious_Bound} to~\eqref{eq:20} to obtain,
for all $\epsilon \geq \alpha_3 \alpha_2^{k-1} n\lambda/\vartheta$,
\begin{equation}\label{eq:23}
  \Pr
  \left(\pi_{k}-\rho^{k}\pi_{0}
    \geq \epsilon
  \right)
  \le \left( \frac{\exp(1)}{\alpha_3}
    \cdot\frac{\vartheta\epsilon}{ n\lambda \alpha_2^{k-1}} \right)^{\alpha_3} 
  \exp \left( -\frac{\vartheta\epsilon}{ n\lambda \alpha_2^{k-1}} \right).
\end{equation}

Next, note that $\min\{\beta/\rho,\,\rho/\beta\}\le1$, and so
from~\eqref{eq:20}, for all $\epsilon \geq k
\alpha_2^{k-1} n\lambda/\vartheta$,
\begin{align}
  \Pr \left(\pi_{k}-\rho^{k}\pi_{0} \geq \epsilon \right)
  &\leq\inf_{\theta\in[0,1/\alpha_1)} \left\{
    \frac{\exp(-\theta\vartheta\epsilon)} {(1-\theta  n\lambda
      \alpha_2^{k-1}\vartheta)^{k}} \right\} \nonumber \\\label{eq:25}
  &\overset{(i)}\leq \left(
    \frac{\exp(1)}{k}\cdot\frac{\vartheta \epsilon}{ n\lambda \alpha_2^{k-1}}
  \right)^{k} \exp \left( -\frac{\vartheta \epsilon}{ n\lambda \alpha_2^{k-1}}
  \right),
\end{align}
where $(i)$ follows from Lemma~\ref{Fact:Tedious_Bound}. Let
$\alphabar_{k}:=\min\{\alpha_3, k \}$. Inequalities~\eqref{eq:23}
and~\eqref{eq:25} can be expressed together, for all $\epsilon \geq
\alphabar_{k} \alpha_2^{k-1} n\lambda/\vartheta$, as
\begin{equation}\label{Main Bound}
      \Pr(\pi\k-\rho^k\pi_0\ge\epsilon) \leq 
      \left( \frac{\exp(1)}{\alphabar_{k}}\cdot\frac{\vartheta\epsilon}{ n\lambda \alpha_2^{k-1}} \right)^{\alphabar_{k}} 
      \exp \left( -\frac{\vartheta\epsilon}{ n\lambda \alpha_2^{k-1}}
      \right).
\end{equation}

Consider the case in which $\epsilon\to\infty$. Then
\begin{equation*}
      \Pr(\pi\k-\rho^k\pi_0\ge\epsilon) \leq 
      \Oscr\left[
        \exp\left(
          -\delta\cdot\frac{\vartheta\epsilon}{\alpha_2^{k-1}}
        \right)
      \right],
\end{equation*}
for some positive $\delta$ independent of $\vartheta$ and
$\alpha_2$. 

Now consider the case in which $k\to\infty$. Take the logarithm of both
sides of~\eqref{Main Bound}:
\[
\log \Pr(\pi\k-\rho^k\pi_0\ge\epsilon) \leq 
\alphabar_{k}\log
\left(
  \frac{\vartheta\epsilon}{\alphabar_{k}  n\lambda\alpha_2^{k-1}}
\right)
+\alphabar_k-\frac{\vartheta\epsilon}{ n\lambda\alpha_2^{k-1}}
=
\Oscr\left(
  -\frac{\epsilon}{\alpha_2^{k-1}}
\right).
\]
This implies~\eqref{eq:1}.
\end{proof}

\begin{bcorollary}[Overwhelming tail bounds]
  \label{overwhelm}
  Suppose that Hypothesis~\ref{as:pop-bounds}.B holds. Take $k$ fixed.
  There exists for all $A>0$ a constant $C_{A}>0$ such that
\[
\Pr\,(\pi\k-\rho^k\pi_0\ge\epsilon) 
\leq
C_{A} \epsilon^{-A}.
\]
Take $\epsilon$ fixed. There exists a constant $C_{A}>0$
such that for all $A>0$,
\[
\Pr\,(\pi\k-\rho^k\pi_0\ge\epsilon)
\leq
C_{A} A^{-k}.
\]
\end{bcorollary}

\begin{proof}
  Because the required result follows from Theorem~\ref{th:tail
    bounds}, we can pick up from the proof of that result. In
  particular, the right-hand side of~\eqref{Main Bound} can be
  equivalently expressed in two ways as
  \begin{align}\nonumber
\left( \frac{\exp(1)}{\alphabar_{k}}\cdot\frac{\vartheta\epsilon}{n\lambda \alpha_{2}^{k-1}} \right)^{\alphabar_{k}} 
      \exp \left( -\frac{\vartheta \epsilon}{n\lambda \alpha_{2}^{k-1}} \right) 
  = \begin{cases}
    \Oscr(1)\cdot\epsilon^{\alphabar_{k}}\exp(-\epsilon\cdot\Oscr(1)) \\
    \exp(\phi_{1}(k))\exp\left(-\exp\left(\phi_{2}(k)\right)\right), 
    \end{cases}
  \end{align}
  where
  \[
  \phi_{1}(k) :=
  \alphabar_{k}\log\left(\vartheta \epsilon\alpha_{2}/\alphabar_{k} \lambda\right)+
  \alphabar_{k}-k\alphabar_{k}\log\alpha_{2}
  \text{and}
  \phi_{2}(k) := \log (\vartheta \epsilon\alpha_{2}/\lambda)-k\log\alpha_{2},
  \]
  and the notation $\BigOh(1)$ stands for positive constants.  The
  result then follows from Lemma~\ref{Terry O}.
\end{proof}

\section{Stochastic and sample average  approximations}
\label{sa}
The results of section~\ref{sec:tail-bounds-sagd} are agnostic to the
source of the gradient errors that are made at each iteration. We
translate these generic results into a sampling policies that yields a
linear convergence rate, both in expectation and with overwhelming
probability.

\begin{figure}[h]
    \centering
    \includegraphics[width=.7\textwidth]{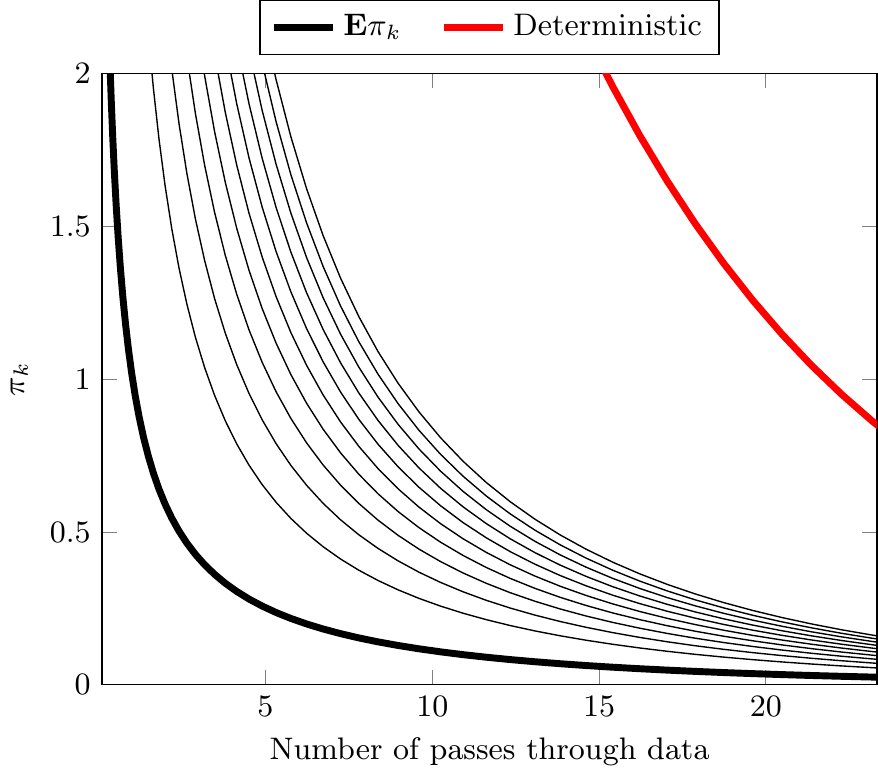}
    \caption{An illustration of the bounds derived in
      Theorem~\ref{thm:sample average bounds}; this figure plots the
      non-asymptotic bound shown in~\eqref{Main Bound}. The thick black
      line (bottom left) shows the bound in expectation (see Part~1 of
      Theorem~\ref{thm:sample average bounds}). For comparison, the
      thick red line (top right) shows the deterministic bound on the
      distance to the solution
      (see~\cite[Theorem~3.1]{FriS:2011}). The thin lines in between
      give the bounds on $\pi\k-\rho^{k}\pi_{0}$ that correspond to
      probabilities $10^{-i}$ for $i=10,20,\ldots,100$. Assume
      $M=300$, $\beta=0.9$, and $\rho=0.9$.}
    \label{fig:tail-bounds}
\end{figure}

\begin{btheorem}[Stochastic-approximation convergence
  rates] \label{as:stochastic approx bounds} Consider the stochastic-approximation
  algorithm described by~\eqref{eq:27} and~\eqref{eq:28A} where
  \[
  \frac{1}{m_{k}} \leq \lambda \beta^{k}
  \]
  for all $k$ for some $\beta\in(0,1)$ and $\lambda>0$. Then the following hold.

  \begin{enumerate}[wide]
    \medskip
  \item{} [Expectation bound] If the variance of the error is bounded, i.e.,
    \[
    \sup_{x}\Ex\|\nabla f(x) - \nabla F(x,Z)\|^2 < \infty,
    \]
    then 
  \[
  \Ex\pi_{k}-\rho^{k}\pi_{0} = \Oscr ([\max\{\beta,\rho\}+\zeta]^{k})
  \text{for all} \zeta>0.
  \]
  If $\rho\ne\beta$,
  then the bound holds with $\zeta = 0$.
    \medskip
  \item{} [Tail bound] If the diameter of the error is bounded, i.e.,
    \[
    \sup_{x}
    \left\{ \sup_{z\in\Omega} [\nabla F(x,z)]_{i} -
    \inf_{z\in\Omega}[\nabla F(x,z)]_{i}
    \right\}
    <\infty,
    \]
  for all $i=1,\ldots,n$, and $\Omega$ is the sample space, then 
  \[
 \Pr(\pi\k-\rho^{k}\pi_{0}\ge\epsilon)=\Oscr
 \left(
   \exp\left[-\frac{\epsilon}{\max\{\beta,\rho\}^{k}}\cdot \zeta\right]
 \right)
 \text{for some} \zeta>0.
 \]
  \end{enumerate}
\end{btheorem}
\begin{proof}
  \paragraph{Part 1 (Expectation Bound)} 
  Because the random variables $Z_{1},\dots,Z_{m_{k}}$ are
  independent copies of $Z$, the expected sample error is equal to the sample average. Thus,
  \begin{align*}
  \Ex \norm{e\k}^{2}
  &= 
  \frac1{m_{k}^{2}}
  \Ex
  \bigg\|
    \sum_{i=1}^{m_{k}} \big[ \nabla f(x\k) - \nabla F(x\k,Z_{i}) \big]
  \bigg\|^{2}  
  = \Ex
     \big\|
       \nabla f(x\k) - \nabla F(x\k,Z)       
     \big\|^{2}/m_{k}
\\&\le \sup_{x}\Ex\|\nabla f(x) - \nabla F(x,Z)\|^2/m_{k}
  \le \lambda \beta^{k},
  \end{align*}
  therefore satisfying Hypothesis~\ref{as:pop-bounds}.A and thus
  the hypothesis of Theorem~\ref{th:bound-in-expectation}.
  \paragraph{Part 2 (Tail Bound)} This follows from Hoeffding's
  Inequality; see Theorem~\ref{thm:Hoeffding Bound}. Thus we satisfy
  Hypothesis~\ref{as:pop-bounds}.B and therefore the hypothesis of
  Theorem~\ref{th:tail bounds}.
\end{proof}

\begin{btheorem}[Sample average gradient convergence
  rates] \label{thm:sample average bounds} Consider the algorithm
  described by~\eqref{eq:27} and~\eqref{eq:28} where
  \begin{equation}\label{eq:5}
  \frac{1}{m_{k}} 
    \left(
      1-\frac{m_{k} -1}{M}
    \right)
  \leq \lambda \beta^{k}
  \end{equation}
  for all $k$ for some $\beta\in(0,1)$ and $\lambda>0$. Then the following hold.

  \begin{enumerate}[wide]
    \medskip
  \item{} [Expectation bound] If the population variance is bounded, i.e.,
    \[
    \sup_{x} \frac{1}{M-1} \sum_{i=1}^{M} \| f(x) - \nabla f_{i}(x)\|^2 < \infty,
    \]
    then 
  \[
  \Ex\pi_{k}-\rho^{k}\pi_{0} = \Oscr ([\max\{\beta,\rho\}+\zeta]^{k})
  \text{for all} \zeta>0.
  \]
  If $\rho\ne\beta$,
  then the bound holds with $\zeta = 0$.
    \medskip
  \item{} [Tail bound] If the population diameter is bounded, i.e.,
    \[
    \sup_{x}
    \left\{ \max_{j}\, [\nabla f_{j}(x)]_{i} -
    \min_{j}\,[\nabla f_{j}(x)]_{i}
    \right\}
    <\infty,
    \] 
  for all $i=1,\ldots,n$, then 
  \[
 \Pr(\pi\k-\rho^{k}\pi_{0}\ge\epsilon)=\Oscr
 \left(
   \exp\left[-\frac{\epsilon}{\max\{\beta,\rho\}^{k}}\cdot \zeta\right]
 \right)
 \text{for some} \zeta>0.
 \]
 \item{} [Deterministic bound] If the diameter of the error is bounded, i.e.,
    \[
    \sup_{x}
    \| f_{i}(x) \|^{2} < \infty 
    \] 
    for all $i=1,\ldots,n$, then
  \[
  \pi_{k}-\rho^{k}\pi_{0} = \Oscr ([\max\{\beta,\rho\}+\zeta]^{k})
  \text{for all} \zeta>0.
  \]
  If $\rho\ne\beta$,
  then the bound holds with $\zeta = 0$.
  \end{enumerate}
\end{btheorem}

\begin{proof}
  \paragraph{Part 1 (Expectation Bound)} 
  Let 
  \[
  S(x) := \frac{1}{M-1} \sum_{i=1}^{M} \| f(x) - \nabla f_{i}(x)\|^2.
  \]
  Then from Friedlander and Schmidt~\cite[\S3.2]{FriS:2011},
  \begin{align*}
    \Ex \norm{e\k}^{2} &= \left( 1 - \frac{m_{k}}{M} \right)
    \frac{S(x_{k})}{m_{k}} 
    \le \left( 1 - \frac{m_{k}-1}{M} \right) \frac{\sup_{x}S(x)}{m_{k}}
    \le \lambda \beta^{k},
  \end{align*}
  therefore satisfying Hypothesis~\ref{as:pop-bounds}.A and thus
  the hypotheses of Theorem~\ref{th:bound-in-expectation}.
  \paragraph{Part 2 (Tail Bound)} This follows from Serfling's
  Inequality; see Theorem~\ref{thm:Serfling Bound}. Thus we satisfy
  Hypothesis~\ref{as:pop-bounds}.B and therefore the hypothesis of
  Theorem~\ref{th:tail bounds}.
  \paragraph{Part 3 (Deterministic Bound)} 
  Refer to Friedlander and Schmidt~\cite[\S3.1]{FriS:2011}.
\end{proof}

The asymptotic notation in the theorem statements helps us simplify
the results, however non asymptotic bounds are available explicitly
within the proofs. Figure~\ref{fig:tail-bounds} illustrates the non
asymptotic bounds~\eqref{na bound} and~\eqref{Main Bound} that
correspond to parts~1 and~2 of Theorem~\ref{thm:sample average
  bounds}; the deterministic bounds follow from Friedlander and
Schmidt~\cite[Theorem 3.1]{FriS:2011}.

\section{Numerical experiments}
\label{sec:numer-exper}

\begin{figure}[t]
  \centering
  \begin{tabular}{r}
      \includegraphics{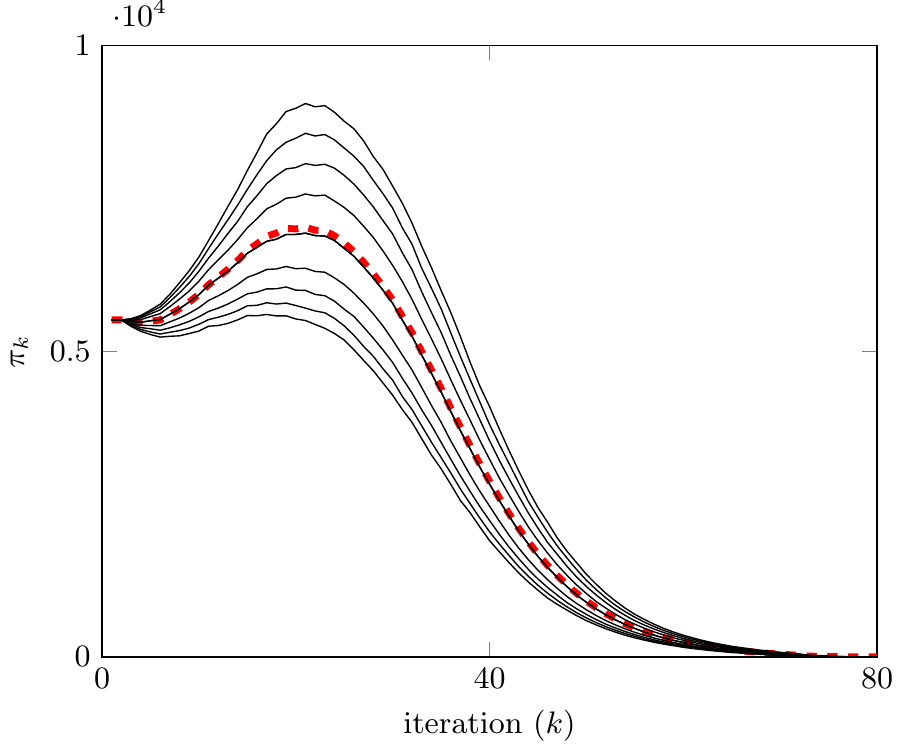}
   \\ \includegraphics{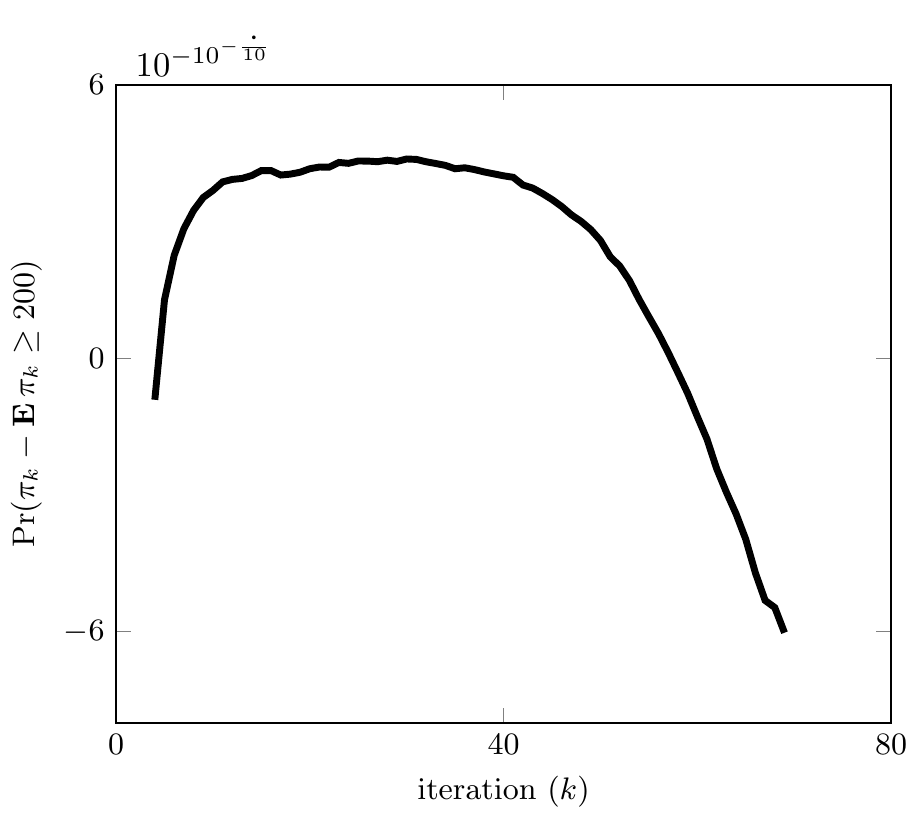}
  \end{tabular}
  \caption{Top panel: distance to solution for quantiles $1-0.5^{j}$
    and $0.5^{j}$, $j=-5:5$. Bottom panel: probability of the deviation
    from expected value against a log-log y-axis, which exhibits the
    tail that converges with a doubly-exponential tail.}
  \label{fig:simulation}
\end{figure}
 
Figure~\ref{fig:simulation} shows the results of a Monte Carlo
simulation on a logistic regression problem, where
\[
f_{i}(x) = \log(1 + \exp[-b_{i}\langle a_{i},x\rangle]),
\]
$a_{i}\in\Real^{n}$ is a vector of input features, and
$b_{i}\in\{-1,1\}$ is the corresponding observation. For this problem,
we generate a dataset with $M=100$ pairs $(a_{i},b_{i})$ of random
points. Algorithm~\eqref{eq:27} and~\eqref{eq:28}, where the sample
size satisfies~\eqref{eq:5} with $\beta \approx .91$, is run 10K times
on this fixed dataset. The starting point between runs is the same,
and the only difference is the randomness of the
sampling. Figure~\ref{fig:simulation} summarizes the results of this
experiment. As expected, the sample paths are concentrated tightly
around the mean. Furthermore, the probability of deviating from the
mean decays doubly-exponentially (cf.~\ref{thm:sample average
  bounds}), as evidenced by the linear tail shown in the bottom panel.

\begin{appendices}
\section{Proof of Lemma~\ref{lem:obj-val-bnd}}
\label{app:proof-of-obj-val-bnd}

In all results of this section, we assume that the sequence $x\k$ is
generated by~\eqref{eq:27}. For this section only, we use the
abbreviation $[\cdot]_{+}:=\prox_{(1/L)g}\set{\cdot}$. The following
result is a simple modification of the ``three-point property''
frequently used in the literature \cite{Tseng:2010}, in order to make
$e\k$ explicit.

\begin{lemma}[Three-point property with error] \label{le:3-points}
For all $y\in\dom g$,
\begin{align*}
  g(y) &\geq g(x_{k+1})
  +\innerp{\nabla f(x_{k})+e_{k}}{x_{k+1}-y}
 \\&\qquad
 +\tfrac{L}{2}\|x_{k+1}-x_{k}\|^{2}
 +\tfrac{L}{2}\|y-x_{k+1}\|^{2}-\tfrac{L}{2}\|y-x_{k}\|^{2}.
\end{align*}
\end{lemma}

\begin{proof}
Let $\psi_{k}(x) := g(x)+f(x_{k})+\left\langle \nabla
  f(x_{k})+e_{k},x-x_{k}\right\rangle +\tfrac{L}{2}\|x-x_{k}\|^{2}$. 
Because $\psi_{k}$ is strongly convex,
\[
  \psi_{k}(y)\geq\psi_{k}(x)
  + \innerp{q}{y-x}
  + \tfrac{L}{2}\|y-x\|^{2}
  \quad\mbox{for all $x,y$ and all $q\in\partial\psi\k(x)$.}
\]
Choose $x=x_{k+1}:=\argmin\,\phi_{k}(x)$. Because
$0\in\partial\psi\k(x\kp1)$, we have
$\psi_{k}(y)\geq\psi_{k}(x_{k+1})+\tfrac{L}{2}\|y-x_{k+1}\|^{2}$,
which, after simplifying, yields the required result.
\end{proof}

\begin{lemma}\label{lem:fourbounds} Let $\bar{x}_k$ be the projection of $x_k$ onto $\Soln$. Then
\begin{subequations}
\begin{align} 
  \|x_{k}-\bar{x}_k\|^{\phantom2}
  &\leq\tau\|x_{k}-x_{k+1}\|+\tfrac{\tau}{L}\|e_{k}\|;
  \label{eq:a1a}
\\ \|x_{k}-\bar{x}_k\|^{2} 
  &\leq2\tau^{2}\|x_{k}-x_{k+1}\|^{2}+\tfrac{5}{4}(\tau^{2}/L^{2})\|e_{k}\|^{2};
  \label{eq:a1b}
\\ \|x_{k+1}-\bar{x}_k\|^{\phantom2} & \leq(1+\tau)\|x_{k}-x_{k+1}\|+\tfrac{\tau}{L}\|e_{k}\|;
  \label{eq:a1c}
\\ \|x_{k+1}-\bar{x}_k\|^{2} & \leq\tfrac{1}{2}[2+5\tau+3\tau^{2}]\|x_{k}-x_{k+1}\|^{2}+\tfrac{1}{2L^{2}}[3\tau^{2}+\tau]\|e_{k}\|^{2}.
  \label{eq:a1d}
\end{align}
\end{subequations}
\end{lemma}
\begin{proof}
\paragraph{Part \eqref{eq:a1a}} For all $k$,
\begin{align*}
\|x_{k}-\bar{x}_k\|
  &\overset{(i)}{\leq}\tau\|x_{k}
  - [x_{k}-\tfrac{1}{L}\nabla f(x_{k})]_{+}\|
\\&\leq
    \tau \|x_{k}-x_{k+1}\|
  + \tau \|x_{k+1}
  - [x_{k}-\tfrac{1}{L}\nabla f(x_{k})]_{+}\|
\\&=\tau \|x_{k}-x_{k+1}\|
  + \tau \|[x_{k}-\tfrac{1}{L}(\nabla f(x)+e_{k})]_{+}
          -[x_{k}-\tfrac{1}{L}\nabla f(x_{k})]_{+}\|
\\& \overset{(ii)}{\leq}\tau\|x_{k}-x_{k+1}\|+\tfrac{\tau}{L}\|e_{k}\|,
\end{align*}
where $(i)$ follows from Assumption~\eqref{a:natural-res} and $(ii)$
follows from the nonexpansiveness of the proximal operator.

\paragraph{Part \eqref{eq:a1b}} Square both sides of~\eqref{eq:a1a}
and then apply the inequality
\begin{equation}\label{eq:crossterm}
ab\leq\frac{a^{2}}{2\alpha}+\frac{\alpha b^{2}}{2},
\quad \forall\alpha>0,
\end{equation}
to bound the cross terms:
\begin{align*}
  \|x_{k}-\bar{x}_k\|^{2}
  & \leq\tau^{2}\|x_{k}-x_{k+1}\|^{2}+(\tau/L)^{2}\|e_{k}\|^{2}+(\tau^{2}/L)\|x_{k}-x_{k+1}\|\|e_{k}\|
\\& \leq\big(\tau^{2}+\tfrac{\tau^{2}\alpha}{2L}\big)\|x_{k}-x_{k+1}\|^{2}+\big(\tfrac{\tau^{2}}{L^{2}}+\tfrac{\tau^{2}}{2L\alpha}\big)\|e_{k}\|^{2}\qquad(\forall\alpha>0)
\\& \leq2\tau^{2}\|x_{k}-x_{k+1}\|^{2}+\tfrac{5}{4}(\tau^{2}/L^{2})\|e_{k}\|^{2}.
\end{align*}

\paragraph{Part \eqref{eq:a1c}} Use the triangle inequality
and~\eqref{eq:a1a}:
\[
  \|x_{k+1}-\bar{x}_k\|
  \leq\|x_{k+1}-x_{k}\|+\|x_{k}-\bar{x}_k\|
  \leq(1+\tau)\|x_{k}-x_{k+1}\|+(\tau/L)\|e_{k}\|.
\]

\paragraph{Part \eqref{eq:a1d}} Square both sides above, and use the
same technique used in Part~\eqref{eq:a1b} to bound the cross-terms:
\[
  \|x_{k+1}-\bar{x}_k\|^{2}
  \leq\tfrac{1}{2}(2+5\tau+3\tau^{2})\|x_{k}
  -x_{k+1}\|^{2}+\tfrac{1}{2L^{2}}(3\tau^{2}+\tau)\|e_{k}\|^{2}.
\]
\end{proof}

\begin{lemma}[Sufficient decrease] \label{lem:sufficient_decrease} For all $k$, 
\[
\pi_{k+1}\leq\left(1-\frac{1}{1+40\tau^{2}}\right)\pi_{k}+\frac{1}{L}\cdot\frac{40\tau^{2}}{1+40\tau^{2}}\|e_{k}\|^{2}.
\]
\end{lemma}

\begin{proof} First, specialize Lemma~\ref{le:3-points} with $y=x_{k}$:
\begin{equation} \label{eq:3pta}
g(x_{k+1}) \leq g(x_{k})-\left\langle \nabla f(x_{k})+e_{k},x_{k+1}-x_{k}\right\rangle -L\|x_{k+1}-x_{k}\|^{2}.
\end{equation}
Then,
\begin{align*}
h(x_{k+1}) & \overset{(i)}{\leq}f(x_{k})+\left\langle \nabla f(x_{k}),x_{k+1}-x_{k}\right\rangle +\tfrac{L}{2}\|x_{k+1}-x_{k}\|^{2}+g(x_{k+1})\\
 & \overset{(ii)}{\leq}f(x_{k})+\left\langle \nabla f(x_{k}),x_{k+1}-x_{k}\right\rangle +\tfrac{L}{2}\|x_{k+1}-x_{k}\|^{2}+g(x_{k})\\
 &\qquad-\left\langle \nabla f(x_{k})+e_{k},x_{k+1}-x_{k}\right\rangle -L\|x_{k+1}-x_{k}\|^{2}\\
 &= h(x_{k})-\left\langle e_{k},x_{k+1}-x_{k}\right\rangle -\tfrac{L}{2}\|x_{k}-x_{k+1}\|^{2}\\
 &\leq h(x_{k})+\tfrac{1}{2\alpha}\|e_{k}\|^{2}+\big(\tfrac{\alpha}{2}-\tfrac{L}{2}\big)\|x_{k}-x_{k+1}\|^{2},
\end{align*}
where $(i)$ uses Assumption~\eqref{a:lipschitz} and $(ii)$ uses
the~\eqref{eq:3pta}. Choose $\alpha=L/2$ and rearrange terms to obtain
the required result.
\end{proof}

We now proceed with the proof of Lemma~\ref{lem:obj-val-bnd}.  Let
$\bar{x}_k$ be the projection of $x_k$ onto the solution set
$\Soln$. By the mean value theorem,
\begin{equation} \label{eq:eq_mv}
  f(x_{k+1})-f(\bar{x}_k)
  = \left\langle \nabla f(\xi),x_{k+1}-\bar{x}_k\right\rangle.
\end{equation}
From Lemma~\ref{le:3-points}, we have
\begin{align} \label{eq:eq_3pt1}
g(x_{k+1})-g(\bar{x}_k) 
& \leq-\left\langle \nabla f(x_{k})+e_{k},x_{k+1}-\bar{x}_k\right\rangle \nonumber \\ 
& \qquad-\tfrac{L}{2}\|x_{k+1}-x_{k}\|^{2}-\tfrac{L}{2}\|\bar{x}_k-x_{k+1}\|^{2}+\tfrac{L}{2}\|\bar{x}_k-x_{k}\|^{2}\nonumber \\
& \leq-\left\langle \nabla f(x_{k})+e_{k},x_{k+1}-\bar{x}_k\right\rangle +\tfrac{L}{2}\|\bar{x}_k-x_{k}\|^{2}.
\end{align}
Also note that
\begin{align*}
\left\langle \nabla f(\xi)-\nabla f(x_{k}),x_{k+1}-\bar{x}_k\right\rangle 
 &\leq\|\nabla f(\xi)-\nabla f(x_{k})\|\|x_{k+1}-\bar{x}_k\|
\\&\overset{(i)}{\leq}L\|\xi-x_{k}\|\|x_{k+1}-\bar{x}_k\|
\\&\leq L[\|x_{k+1}-x_{k}\|+\|x_{k}-\bar{x}_k\|]\cdot\|x_{k+1}-\bar{x}_k\|
\\&\leq[L(1+\tau)\|x_{k}-x_{k+1}\|+\tau\|e_{k}\|]
\\&\qquad\cdot[(1+\tau)\|x_{k}-x_{k+1}\|+\tfrac{\tau}{L}\|e_{k}\|]
\\&= L(1+\tau)^{2}\|x_{k}-x_{k+1}\|^{2}
\\& \qquad +2[\tau(1+\tau)]\|x_{k}-x_{k+1}\|\|e_{k}\|+\tau^{2}/L\|e_{k}\|^{2}
\\&\leq[L(1+\tau)^{2}+\tfrac{1}{\alpha}\tau(1+\tau)]\|x_{k}-x_{k+1}\|^{2}
\\&\qquad+[\tau^{2}/L+\alpha\tau(1+\tau)]\|e_{k}\|^{2}
\\&\leq L(1+3\tau+2\tau^{2})\|x_{k}-x_{k+1}\|^{2}+\tfrac{1}{L}(2\tau^{2}+\tau)\|e_{k}\|^{2},
\end{align*}
where $(i)$ follows from~\eqref{a:lipschitz}. In the steps which
follow, we apply the relevant inequalities in
Lemma~\ref{lem:fourbounds}, group terms, bound every cross term
using~\eqref{eq:crossterm}, and repeat the process until we reach the
final result:
\begin{align*}
h(x_{k+1})-h(\bar{x}_k)
&\overset{(i)}\leq\left\langle \nabla f(\xi)-\nabla f(x_{k}),x_{k+1}-\bar{x}_k\right\rangle -\left\langle e_{k},x_{k+1}-\bar{x}_k\right\rangle +\tfrac{L}{2}\|\bar{x}_k-x_{k}\|^{2}\\
 & \leq L(1+3\tau+2\tau^{2})\|x_{k}-x_{k+1}\|^{2}+\tfrac{1}{L}(2\tau^{2}+\tau)\|e_{k}\|^{2}\\
 & \qquad+\tfrac{\alpha}{2}\|e_{k}\|^{2}+\tfrac{1}{2\alpha}\|x_{k+1}-\bar{x}_k\|^{2}+\tfrac{L}{2}\|\bar{x}_k-x_{k}\|^{2} \qquad \forall \alpha >0\\
 & \leq L(1+3\tau+2\tau^{2})\|x_{k}-x_{k+1}\|^{2}+\tfrac{1}{L}(2\tau^{2}+\tau)\|e_{k}\|^{2}\\
 & \qquad+\tfrac{\alpha}{2}\|e_{k}\|^{2}+\tfrac{1}{4\alpha}[2+5\tau+3\tau^{2}]\|x_{k}-x_{k+1}\|^{2}\\
 & \qquad +\tfrac{1}{4L^{2}\alpha}[3\tau^{2}+\tau]\|e_{k}\|^{2}+L\tau^{2}\|x_{k}-x_{k+1}\|^{2}+\tfrac{5L}{8}(\tau^{2}/L^{2})\|e_{k}\|^{2}\\
 & \leq\left(L(1+3\tau+2\tau^{2})+\tfrac{1}{4\alpha}[2+5\tau+3\tau^{2}]+L\tau^{2}\right)\|x_{k}-x_{k+1}\|^{2}+\\
 & \qquad\left(\tfrac{1}{L}(2\tau^{2}+\tau)+\tfrac{\alpha}{2}+\tfrac{1}{4L^{2}\alpha}[3\tau^{2}+\tau]+\tfrac{5L}{8}(\tau^{2}/L^{2})\right)\|e_{k}\|^{2}\\
 & \overset{(ii)}{\leq}10L\tau^{2}\|x_{k}-x_{k+1}\|^{2}+\tfrac{1}{L}10\tau^{2}\|e_{k}\|^{2}\\
 & \overset{(iii)}{\leq}40\tau^{2}[h(x_{k})-h(x_{k+1})]+(4/L^2+\tfrac{1}{L}10\tau^{2})\|e_{k}\|^{2}\\
  &\leq40\tau^{2}[h(x_{k})-h(x_{k+1})]+\tfrac{1}{L}40\tau^{2}\|e_{k}\|^{2}.
\end{align*}
In the steps above, $(i)$ follows by add inequalities~\eqref{eq:eq_mv}
and~\eqref{eq:eq_3pt1}. Also, we make use of
Lemma~\ref{lem:fourbounds} to bound all stray terms in terms of
$\|x_k-x_{k+1}\|^2$ and $\|e_k\|^2$, and Equation~\eqref{eq:crossterm}
to bound the cross-terms. In $(ii)$ we make use of the assumption that
$\tau\geq1$ and set $\alpha=1/L$. Finally, in $(iii)$ we make use of
Lemma~\ref{lem:sufficient_decrease} to transition from a bound on the
distance between successive iterates $x_k$ to differences in
successive values $h(x_k)$.  Rearranging terms, we get
\[
(1+40\tau^{2})h(x_{k+1})-(1+40\tau^{2})h(\bar{x}_k)\leq40\tau^{2}(h(x_{k})-h(\bar{x}_k))+\tfrac{1}{L}40\tau^{2}\|e_{k}\|^{2},
\]
which is true if and only if the desired result holds:
\[
\pi_{k+1}\leq\left(1-\frac{1}{1+40\tau^{2}}\right)\pi_{k}+\frac{1}{L}\cdot\frac{40\tau^{2}}{1+40\tau^{2}}\|e_{k}\|^{2}.
\]

\section{Auxiliary results}

\ 
\begin{lemma}\label{Terry O} 
  Suppose that
  \begin{align*}
    \phi_{1}(k)&=\Oscr(k^{\Oscr(1)})\exp(-\Oscr(k^{\Oscr(1)})),
  \\\phi_{2}(k)&=\exp(\Oscr(k^{\Oscr(1)}))\exp(-\exp(\Oscr(k^{\Oscr(1)}))),
  \end{align*}
  where $\Oscr(1)$ stands for positive constants. Then for each $A>0$ there exists a positive constant $C_{A}$ such
  that
  \begin{align}
    \phi_{1}(k)&\leq C_{A} k^{-A}, \label{eq:22}
  \\\phi_{2}(k)&\leq C_{A} A^{-k}. \label{eq:24}
  \end{align}
\end{lemma}

\begin{proof}
  The statement follows by taking the logarithms on both sides
  of~\eqref{eq:22} and~\eqref{eq:24}.
\end{proof}

\begin{lemma} \label{fact:Q-Pochhammer_Lower_Bound}
  For $y\in(0,1)$ and $x\in[0,1]$,
\begin{equation} \label{eq:11}
(1-x)^{1-1/\log y}\leq\prod_{i=0}^{\infty}(1-xy^{i}).
\end{equation}
\end{lemma}
\begin{proof}
  To prove the lower bound, we use the following fact:
  \begin{equation*} 
    \ln(1-x)  \geq-\frac{x}{1-x}
    \quad
    \mbox{for all} \quad x\in[0,1).
  \end{equation*}
  Therefore,
\begin{align*}
\prod_{i=1}^{\infty}(1-xy^{i}) & =\exp\left(\sum_{i=1}^{\infty}\log\left(1-xy^{i}\right)\right)\\
 & \geq\exp\left(\sum_{i=1}^{\infty}-\frac{y^{i}}{1/x-y^{i}}\right)\\
 & \geq\exp\left(-\int_{0}^{\infty}\frac{y^{i}}{1/x-y^{i}}di\right)\\
 & =\exp\left(-\frac{\log(1-x)}{\log(y)}\right)
  \geq(1-x)^{-1/\log y}.
\end{align*}
Thus,
\[
\prod_{i=0}^{\infty}(1-xy^{i})  =(1-x)\prod_{i=1}^{\infty}(1-xy^{i})
\ge(1-x)^{1-1/\log y},
\]
as required.
\end{proof}

\begin{lemma}
For $y\in(0,1)$ and $x\in[0,1]$,
\[
\exp\left(-\frac{\log(1-x/y)-\log(1-xy^{N+1})}{\log(y)}\right)\leq\prod_{i=0}^{N}(1-xy^{i}).
\]
\end{lemma}
\begin{proof}
Similar to the proof of the previous inequality
\begin{align*}
\prod_{i=1}^{N}(1-xy^{i}) & =\exp\left(\sum_{i=1}^{N}\log\left(1-xy^{i}\right)\right)\\
 & \geq\exp\left(\sum_{i=1}^{N}-\frac{xy^{i}}{1-xy^{i}}\right)\\
 & \geq\exp\left(-\int_{0}^{N}\frac{xy^{i}}{1-xy^{i}}di\right)\\
 & \geq\exp\left(-\frac{\log(1-x)-\log\left(1-xy^{N}\right)}{\log(y)}\right).
\end{align*}
Thus,
\begin{align*}
\prod_{i=0}^{N}(1-xy^{i}) & =\prod_{i=1}^{N+1}(1-(x/y)y^{i})\\
 & \geq\exp\left(-\frac{\log(1-x/y)-\log(1-xy^{N+1})}{\log(y)}\right),
\end{align*}
as required.
\end{proof}

\begin{lemma} \label{Fact:Tedious_Bound} Let
  $k>0$, $\mu>0$, and $\epsilon>0$. Then for $y\in(0,1)$ and $x\in(0,1]$,
\[
\inf_{\theta>0}\left\{
  \exp(-\theta\epsilon\nu)\prod_{i=0}^{N-1}\left(1-\theta
    xy^{i}\right)^{-k}\right\} \leq
\left(\frac{\exp(1)}{\alpha}\cdot \frac{\epsilon\nu}{x}\right)^{\alpha}\exp\left(-\frac{\epsilon\nu}{x}\right),
\]
where $\alpha=\frac{1}{k}\left(\frac{1}{\log(1/y)}+1\right)$.
\end{lemma}
\begin{proof}
  By inverting both sides of~\eqref{eq:11} we obtain the following
  inequality
\begin{equation} \label{eq:Q--LB}
  \prod_{i=0}^{\infty}(1-xy^{i})^{-k}
  \leq
  \exp\left(-\log(1-x)\left[\frac{1}{\log(1/y)}+1\right]\right).
\end{equation}
Therefore, for $\epsilon\geq\alpha x/v$,
\begin{align*}
  &\hspace*{-.5in}\inf_{\theta>0} \left\{
    \exp(-\theta\epsilon\nu)\prod_{i=0}^{N-1}(1-\theta xy^{i})^{-k}
  \right\}
  \\
  &\leq\inf_{\theta>0} \left\{
    \exp(-\theta\epsilon\nu)\prod_{i=0}^{\infty}(1-\theta xy^{i})^{-k}
  \right\}
  \\
  &\overset{(i)}{\leq}\inf_{\theta>0} \left\{ \exp \left( -\frac{1}{k}
      \left[ \frac{1}{\log(1/y)}+1 \right]\log \left( 1-\theta x
      \right) -\theta v\epsilon \right) \right\}
  \\
  &=\inf_{\theta>0} \left\{ \exp \left( -\alpha\log \left( 1-\theta x
      \right) -\theta\epsilon\nu \right) \right\}
  \\
  &\overset{(ii)}{=}\exp \left( -\alpha\log
    \left(1- \left( \frac{1}{x}-\frac{\alpha}{v\epsilon} \right)x
    \right)- \left( \frac{1}{x}-\frac{\alpha}{v\epsilon} \right)
    v\epsilon \right)
  \\
  &= \left(\frac{\exp(1)}{\alpha}\cdot
    \frac{\epsilon\nu}{x}\right)^{\alpha}\exp\left(-\frac{\epsilon\nu}{x}\right),
\end{align*}
where $(i)$ follows from~\eqref{eq:Q--LB}; and $(ii)$ uses the
substitution $\theta=1/x-\alpha/v\epsilon$, which can be shown to be
the optimal choice of $\theta$. Because $\theta>0$, $\epsilon>\alpha
x/v$.
\end{proof}

For the remainder of this section, define the sample average 
\[
S_{m} := \frac{1}{m}\sum_{i}^{m}X_{i}
\]
for a sequence of random variables $\{X_{1},\ldots,X_{m}\}$.
\begin{theorem}[Hoeffding~{\cite[Theorem~2]{Hoeffding:1963}}]
  \label{thm:Hoeffding Bound} Consider independent random variables
  $\{X_{1},\ldots,X_{m}\}$, $X_{i}:\Omega\rightarrow\Re$. If the
  random variables are bounded, i.e.,
  \[
  d := \sup_{\omega\in\Omega}X_{i}(\omega) - \inf_{\omega\in\Omega}X_{i}(\omega)
  \] is finite, then
  \[
  \Pr\left(S_{m} - \Ex S_{m}\geq\epsilon\right)
  \leq\exp\big({-\epsilon^{2}/\eta_{m}}\big),
  \textt{where}
  \eta_{m} = d^{2}/({2m}).
  \]
\end{theorem}

\begin{theorem}[Serfling~{\cite[Corollary~1.1]{Serfling:1974}}]
  \label{thm:Serfling Bound}
  Let $x_1,\ldots,x_{M}$ be a population, $\{X_{1},\ldots,X_{m}\}$ be
  samples drawn without replacement from the population, and let $ d
  := \max_i x_{i} - \min_i x_{i}.  $ Then
  \[
  \Pr\left(S_{m}-\Ex S_{m} \geq\epsilon\right)
  \leq\exp\big({-\epsilon^{2}/\eta_m}\big),
  \textt{where}
  \eta_{m} = \frac{d^{2}}{2m}\left(1 - \frac{m-1}{M} \right).
  \]
\end{theorem}

Because $\eta_m$ is strictly decreasing in $m$, the Serfling bound is
uniformly better than the Hoeffding bound. Note that the Serfling
bound is not tight: in particular, when $M=m$ (i.e., $S_{m}=\Ex S_{m}$), the
bound is not zero (except for degenerate population).

\end{appendices}
\section*{Acknowledgements}
We are grateful to our colleagues Bill Aiello and Ting Kei Pong for
invaluable comments. We are also indebted to Jason Swanson for
generously answering our questions regarding conditional expectations,
and for his excellent lecture
notes~\cite{Swanson:2013:cond-layer-cake} on the topic.

\bibliography{master,friedlander}

\end{document}